\documentclass[reqno, 12pt]{amsart}%
\usepackage{amsaddr}
\usepackage{amssymb}
\usepackage[nesting]{hyperref}
\usepackage[pdftex]{graphicx}
\usepackage{listings}
\usepackage{multirow}
\usepackage{placeins}
\usepackage{color}
\usepackage{subfigure}
\usepackage{lscape}
\usepackage{dsfont}
\usepackage{amsmath}
\usepackage{amsfonts}
\usepackage{tikz}
\usepackage{verbatim}
\usepackage{accents} 
\usepackage{todonotes}
\usepackage{mathtools}
\usepackage{bbm}
\usepackage{algorithm}
\usepackage{algpseudocode}
\usepackage{subcaption}
\usepackage{todonotes}

\usepackage[style=authoryear, natbib=true, backend=biber]{biblatex}
\usepackage{csquotes}
\usepackage{mathrsfs}
\usepackage{cleveref}
\let\cref\Cref 
\usepackage{physics}
\usepackage{enumerate}
\usepackage{bbm}

\newcommand{\BlackBoxes}{\global\overfullrule5pt}

\BlackBoxes

\providecommand{\U}[1]{\protect\rule{.1in}{.1in}}

\allowdisplaybreaks

\newcommand{\R}{\mathbb{R}} 
\newcommand{\N}{\mathbb{N}} 

\newcommand{\PP}{\mathbb{P}}
\newcommand{\EE}{\mathbb{E}}

\newcommand{\E}{\mathbb{E}}

\newcommand{\1}{\mathbf{1}}

\newcommand{\ThetaSet}{\Theta}
\newcommand{\Pmax}{P_{\max}}

\newcommand{\uth}{\underline{\vartheta}}
\newcommand{\oth}{\overline{\vartheta}}

\newtheorem{theorem}{Theorem}
\newtheorem{corollary}[theorem]{Corollary}
\newtheorem{lemma}[theorem]{Lemma}

\theoremstyle{definition}

\newtheorem{remark}[theorem]{Remark}
\newtheorem{definition}[theorem]{Definition}
\newtheorem{assumptions}[theorem]{Assumption}

\crefname{theorem}{Theorem}{Theorems}
\crefname{corollary}{Corollary}{Corollaries}
\crefname{lemma}{Lemma}{Lemmas}
\crefname{proposition}{Proposition}{Propositions}
\crefname{example}{Example}{Examples}
\crefname{remark}{Remark}{Remarks}
\crefname{definition}{Definition}{Definitions}
\crefname{assumptions}{Assumption}{Assumptions}

\numberwithin{equation}{section} \numberwithin{theorem}{section}
\def\0{\kern0pt\-\nobreak\hskip0pt\relax}

\makeatletter
\AtBeginDocument{ \def\@serieslogo{ \vbox to\headheight{ \parindent\z@ \fontsize{6}{7\p@}\selectfont
\vss}}}

\def\makeoverbar#1#2#3#4#5#6#7{ \setbox0=\hbox{$\m@th#2\mkern#5mu{{}#3{}}\mkern#6mu$} 
\setbox1=\null \dimen@=#4\fontdimen8#13 \dimen@=3.5\dimen@
\advance\dimen@ by \ht0 \dimen@=-#7\dimen@ \advance\dimen@ by \wd0
\ht1=\ht0 \dp1=\dp0 \wd1=\dimen@
\dimen@=\fontdimen8#13 \fontdimen8#13=#4\fontdimen8#13
\rlap{\hbox to \wd0{$\m@th\hss#2{\overline{\box1}}\mkern#5mu$}}
\fontdimen8#13=\dimen@}
\def\mylabel#1#2{{\def\@currentlabel{#2}\label{#1}}}
\makeatother
\newcommand{\Exv}{\EE}
\newcommand{\Pmass}{\PP}

\DeclareMathOperator{\proj}{proj}

\newcommand{\statespace}{\mathcal{X}}

\newcommand{\polspace}{\Pi}

\newcommand{\timeindex}{\mathcal{T}}
\newcommand{\TimeIndex}{\mathcal{T}}

\newcommand{\PufferStoremax}{{P_\mathrm{max}}}
\newcommand{\pin}{p_\mathrm{in}}
\newcommand{\pout}{p_\mathrm{out}}
\newcommand{\pem}{p_\mathrm{em}}
\newcommand{\psum}{p_\mathrm{sum}}

\DeclareMathOperator*{\argmax}{arg\,max}

\DeclareMathOperator{\AC}{AC}
\DeclareMathOperator{\ACD}{ACD}
\DeclareMathOperator{\Lip}{Lip}
\DeclareMathOperator{\LC}{L^1}
\DeclareMathOperator{\LCLoc}{L_\mathrm{loc}^1}

\newcommand*\diff{\mathop{}\!\mathrm{d}}

\newcommand{\initdist}{\rho}

\newcommand{\thsspace}{\Theta}
\newcommand{\ths}{\vartheta}

\newcommand{\lt}{\underline{\vphantom{\overline{\ths}}\ths}}
\newcommand{\ut}{\overline{\vphantom{\underline{\ths}}\ths}}
\newcommand{\ltt}{\lt_t}
\newcommand{\utt}{\ut_t}
\newcommand{\thch}[1]{{\lt^*_{#1}}}
\newcommand{\thds}[1]{{\ut^*_{#1}}}
\newcommand{\tthch}[1]{{\lt^\times_{#1}}}
\newcommand{\tthds}[1]{{\ut^\times_{#1}}}

\newcommand{\hch}[1]{\underline{\vphantom{\overline{h}}h}_{#1}}
\newcommand{\hds}[1]{\overline{\vphantom{\underline{h}}h}_{#1}}
\newcommand{\hold}[1]{h_{#1}}

\newcommand{\thcht}{\thch{t}}
\newcommand{\thdst}{\thds{t}}
\newcommand{\tthcht}{\tthch{t}}
\newcommand{\tthdst}{\tthds{t}}

\newcommand{\hcht}{\hch{t}}
\newcommand{\hdst}{\hds{t}}
\newcommand{\holdt}{\hold{t}}

\newcommand{\clip}[3]{\left[#1\right]_{#2}^{#3}}
\newcommand{\clipstorage}[1]{\left[#1\right]^{\PufferStoremax}_{0}}

\newcommand{\pol}{\pi}

\newcommand{\ri}[1]{r^{(1)}_{#1}}
\newcommand{\rii}[1]{r^{(2)}_{#1}}
\newcommand{\lo}{\ell}
\newcommand{\lsample}{\hat{\lo}}
\newcommand{\rsample}{\hat{r}}

\newcommand{\Vsample}{\widehat{V}}
\newcommand{\hsample}{\widehat{h}}
\newcommand{\asample}{\hat{a}}
\newcommand{\xsample}{\hat{x}}

\newcommand{\episode}{e}
\newcommand{\Episode}{E}
\newcommand{\episodespace}{\mathcal{E}}

\newcommand{\Vpara}[1]{V^\ths_{#1}}
\newcommand{\hpara}[1]{h^\ths_{#1}}
\newcommand{\apara}[1]{a^\ths_{#1}}
\newcommand{\hchpara}[1]{\underline h^\ths_{#1}}
\newcommand{\hdspara}[1]{\overline h^\ths_{#1}}
\newcommand{\achpara}[1]{\underline a^\ths_{#1}}
\newcommand{\adspara}[1]{\overline a^\ths_{#1}}

\newcommand{\lsmp}{\lsample^{\episode}}
\newcommand{\Vsmp}[1]{\Vsample^{\ths,\episode}_{#1}}
\newcommand{\hsmp}[1]{\hsample^{\ths,\episode}_{#1}}

\newcommand{\riismp}[1]{\rsample^{(2),\episode}_{#1}}

\newcommand{\gradientsmp}{\widehat{\nabla}}
\newcommand{\asmp}[1]{\asample^{\ths,\episode}_{#1}}
\newcommand{\xsmp}[1]{\xsample^{\ths,\episode}_{#1}}
\newcommand{\lrv}{\lsample^{\Episode}}
\newcommand{\Vrv}[1]{\Vsample^{\Episode,\ths}_{#1}}
\newcommand{\hrv}[1]{\hsample^{\Episode,\ths}_{#1}}
\newcommand{\arv}[1]{\asample^{\Episode,\ths}_{#1}}

\newcommand{\initstate}{x_0}
\newcommand{\dist}[1]{z_{#1}}
\newcommand{\Dist}[1]{Z_{#1}}

\newcommand{\goodset}[1]{B^\ths_{#1}}
\newcommand{\goodsettilde}{\widetilde{B}}
\newcommand{\forbiddenset}[1]{C^\ths_{#1}}
\newcommand{\forbiddensettilde}[1]{\widetilde{C}^\ths_{#1}}

\newcommand{\criticalset}[1]{\{0,\lt_{#1}, \ut_{#1}, \PufferStoremax\}}

\newcommand{\past}[1]{{< #1}}
\newcommand{\future}[1]{{> #1}}

\newcommand{\thspa}[1]{\ths_\past{#1}}
\newcommand{\thsfu}[1]{\ths_\future{#1}}

\usepackage{xspace}
\newcommand{\PLK}{P\L K\xspace}

\newcommand{\defi}[1]{\emph{#1}}

\usepackage{todonotes}

\begin{document}

\title[Heat Storage Optimization]{Optimality of Threshold Policies for Heat Storage Optimization and a Policy Gradient Algorithm}
\author[N. \smash{B\"auerle}]{Nicole B\"auerle$^*$}
\address[N. B\"auerle]{Department of Mathematics,
Karlsruhe Institute of Technology (KIT), D-76128 Karlsruhe, Germany, \href{mailto:nicole.baeuerle@kit.edu}{nicole.baeuerle@kit.edu}, ORCID 0000-0003-0077-3444\footnote{$^*$The authors gratefully acknowledge financial support from the Federal Ministry of Research, Technology and Space (BMFTR), Project-ID 05M22VKB }}


\author[F. \smash{D\"ottling}]{Florian D\"ottling$^*$}
\address[F. D\"ottling]{Department of Mathematics,
Karlsruhe Institute of Technology (KIT), D-76128 Karls\-ruhe, Germany, \href{mailto:florian.doettling@kit.edu} {florian.doettling@kit.edu}}

\makeatletter \providecommand\@dotsep{5} \makeatother

\renewcommand*{\arraystretch}{1.5}

\maketitle


\begin{center}
 \today
\end{center}

    \\

\noindent
{\bf Abstract. }
We study the optimal control of a stylized finite-horizon heat-storage system operated by a prosumer who can buy heat from and sell heat to a district heating network. Heat production and consumption are uncertain, and charging and discharging decisions have to be taken before the subsequent heat balance is observed. The resulting problem is formulated as a Markov decision process with continuous state and action spaces. We first establish a structural characterization of the optimal policy as a two-parameter threshold policy. Under additional regularity assumptions, we prove differentiability and Lipschitz properties of the value functions.

Exploiting the threshold structure, we parameterize the policy by the time-dependent lower and upper thresholds and derive a pathwise stochastic-gradient estimator. We then propose a projected stochastic policy-gradient algorithm, establish a constrained Kurdyka-\L ojasiewicz inequality  and derive global finite-sample bounds for both the objective optimality gap and the distance to the optimal threshold vector. Finally, numerical experiments for a non-stationary heat-balance process illustrate the evolution of the threshold parameters. 

\section{Introduction}
\noindent 
The increasing importance of renewable energy sources, particularly photovoltaic systems, makes the optimal management of the storage device under random supply and demand behavior an important task.

In this work, we  consider a single family household with local heat production (e.g.\ solar thermal panels), energy (heat) storage (e.g.\ hot water tank, battery) and a connection to a district heating network (see \cref{fig:prosumer_model}). The aim is to optimize the management of the energy storage device. We view this problem as a prosumer storage problem (\enquote{prosumer} is a blend word of \enquote{producer} and \enquote{consumer}) and model a simplified version. Using techniques from the theory of Markov decision processes (MDPs) we derive the structure of the optimal policy for the management of the energy storage and derive a policy gradient algorithm to find an optimal policy numerically.

More precisely, we consider the problem in discrete time (on an hourly basis) with a continuous state which is the current charge of the storage. The production of energy and its consumption within an hour are random quantities. The balance of heat production minus consumption is modeled by a sequence of independent, time-inhomogeneous random variables. The role of the storage device is to store energy from the solar panel which is not needed for periods where there is no production, but demand for energy. On top of this, energy can be sold to and bought from the district heating network (DHN). Here we assume that selling and buying prices are deterministic, fixed quantities. The crucial feature of the energy storage device that we model is a decision/implementation delay. Typically, charging and discharging the energy storage may come with time delays, since there are technical restrictions on how fast these charging-/discharging processes are. But this implies a non-trivial decision because in times of high energy production the device has to decide how much energy to sell to the DHN before knowing the energy production and the demand within the next hours. This could lead to a low storage when energy is desperately needed and has to be provided by an emergency heat generator at a higher (emergency) price. A somewhat symmetric decision concerns the time point and amount of buying  energy from the DHN which is after some time delay stored in the local storage. Since this is the focus of our study we simplify other technical considerations and consider a stylized model. 

The aim is to maximize the expected reward from storage decisions over a finite planning horizon. Annual operation at an hourly resolution motivates the consideration of long horizons $T$. The theoretical analysis allows for time-inhomogeneous disturbances. We first prove that an optimal policy is characterized by two time-dependent thresholds: heat energy is bought from the DHN when the storage level is below the lower threshold and sold to the DHN when it exceeds the upper threshold. Both the lower and the upper threshold are time-dependent, so this reduces the policy optimization to a problem with $2T$ parameters. We use a policy-gradient algorithm to approximate these threshold parameters. For the analysis of the policy-gradient algorithm we use the framework provided in \cite{chen2026LandscapePolicyOptimization}, in particular,  we use the Polyak-\L ojasiewicz-Kurdyka  (\PLK) condition to derive theoretical guarantees. In a numerical experiment we consider a 24-hour example with a daily cycle. 

The model considered in this work is stylized. There are several aspects of this model that are not reflected in reality. These include a linear charging/discharging dynamic with no constraints on the amount charged/discharged, perfect heat storage with no losses, no aging or maintenance of the system, fixed and proportional prices and stochastic independence of the heat balance random variables. While these are significant restrictions, our goal with this work is the establishment of a mathematically founded but structurally simple optimal policy, which could serve as a benchmark solution in more complicated settings or an easily implementable solution in a system with limited sensors. Adding more realistic assumptions would either complicate the threshold structure or destroy it altogether.

{\em Related literature.}
The optimal management of energy storage systems under uncertainty has been studied using a variety of stochastic control and optimization methods. \cite{grillo2015optimal}, for example,
formulate the scheduling of energy storage in power distribution networks with renewable generation as a Markov decision process and solve the resulting finite-state problem by stochastic dynamic programming. For end-user battery storage, \cite{van2013optimal} derive structural properties of the optimal control and show that the optimal policy has a threshold structure. There is also a substantial literature on the control of thermal energy storage systems. Early work by \cite{HenzeDodierKrarti1997} develops a predictive optimal controller for thermal energy storage, while 
\cite{HenzeSchoenmann2003} investigate reinforcement-learning control of
thermal storage systems. 
\cite{QuaggiottoVivianZarrella2021} study the operation of a district heating network with and without thermal storage using model predictive control. More recently, \cite{DillerEtAl2024} develop a dynamic
programming approach for the optimal control of a cascaded heat-pump system with thermal energy storage. 
\cite{ThorvaldsenBackeFarahmand2024} use stochastic dynamic programming for long-term operational planning of residential buildings with seasonal thermal energy storage. Particularly close to the application considered here, \cite{TakamWunderlich2025} formulate the cost-optimal
management of a residential heating system with geothermal energy storage under uncertain heat supply and demand as a stochastic optimal control problem and, after discretization, as a Markov decision process. In contrast to these works, our focus is on deriving an explicit structural characterization of the optimal policy in a stylized continuous-state and continuous-action storage model and subsequently exploiting this structure for policy optimization.

Threshold and control-band policies arise naturally in storage, inventory, and cash-balance problems. Besides the threshold structure established for battery storage by \cite{van2013optimal}, \cite{WalkerAbediKwon2022} design threshold-based storage policies
within a two-stage stochastic programming framework, and
\cite{HuDefourny2022} study price-threshold policies for battery operation under price uncertainty and battery aging. Related structural results are well known for commodity and gas storage; see, e.g.,
\cite{secomandi2010optimal} and  \cite{bauerle2016gas}. A particularly close mathematical analogy is provided by stochastic cash-balance models. 
\cite{EppenFama1969} study dynamic cash-balance problems with proportional costs for upward and downward adjustments, and 
\cite{FeinbergLewis2007} establish structural results for stochastic
cash-balance Markov decision processes, including threshold policies.
The two-sided lower/upper control-band structure obtained below can be viewed as an energy-storage counterpart of this classical cash-balance structure.

Our numerical method is related to the literature on policy-gradient methods for stochastic control and reinforcement learning. 
\cite{AgarwalEtAl2021} give general convergence and approximation results for policy-gradient methods in discounted Markov decision processes. \cite{BhandariRusso2024} identify structural conditions
under which non-convex policy-gradient objectives have no suboptimal
stationary points and satisfy gradient-dominance conditions. In a
finite-horizon stochastic LQ setting, 
\cite{HamblyXuYang2021} prove global linear convergence and also analyze a projected policy-gradient method for constrained policy classes. \cite{ReisingerStockingerZhang2023} study convergence of
policy-gradient-type methods for finite-horizon continuous-time stochastic control problems. Most closely related to our convergence analysis is \cite{chen2026LandscapePolicyOptimization}, who establish a general framework for the landscape of finite-horizon policy optimization with general state and action spaces and derive global convergence results under suitable Polyak-\L ojasiewicz-Kurdyka conditions. We verify and adapt these conditions for the threshold-policy class arising in the present storage problem.

{\em Our contribution.}
We consider a stylized finite-horizon prosumer heat-storage problem with
random heat supply and demand and a decision/implementation delay. The state and action spaces are continuous, and buying from and selling to the district heating network incur different proportional but fixed prices, while unmet demand is covered at a higher emergency price. Our first contribution is a structural result: we prove that in this model an optimal policy is characterized at every time by a lower and an upper storage threshold. Below the lower threshold it is optimal to buy energy and raise the storage level to the lower threshold; above the upper threshold it is optimal to sell energy and reduce the storage level to the upper threshold; between the two thresholds no transaction is made.
Under additional regularity assumptions on the heat-balance distributions, we further establish continuous differentiability of the value functions and Lipschitz continuity of their derivatives.

Our second contribution is an optimization method that exploits this structural result. Instead of optimizing over a general policy space, we parameterize the policy by the $2T$ time-dependent lower and upper thresholds and derive a pathwise stochastic-gradient estimator for the resulting finite-dimensional objective. We propose a projected stochastic policy-gradient algorithm and show, under explicit assumptions on the state and disturbance distributions, that the policy objective satisfies a constrained Polyak-\L ojasiewicz-Kurdyka inequality. This yields global finite-sample bounds for the expected optimality gap. In addition, a quadratic-growth estimate yields a corresponding mean-square bound for the distance between the computed threshold vector and the globally optimal threshold vector. Hence, by choosing the batch size and the number of policy-gradient iterations sufficiently large, both the objective error and the threshold error can be made arbitrarily small. 

{\em Outline of the paper.}
In the next section we formulate the storage problem as a Markov decision process. \cref{sec:optimal_policy} establishes the threshold structure of an optimal policy and regularity properties of the value functions. \cref{sec:algorithm} introduces the projected stochastic policy-gradient algorithm and derives finite-sample guarantees. \cref{sec:numerics} studies a 24-hour Gaussian example. 
The appendix contains auxiliary results and proofs

\section{The prosumer heat storage problem as a Markov decision process}
\begin{figure}
    \begin{minipage}[t]{0.39\textwidth}
        \centering
        \includegraphics[width=\textwidth]{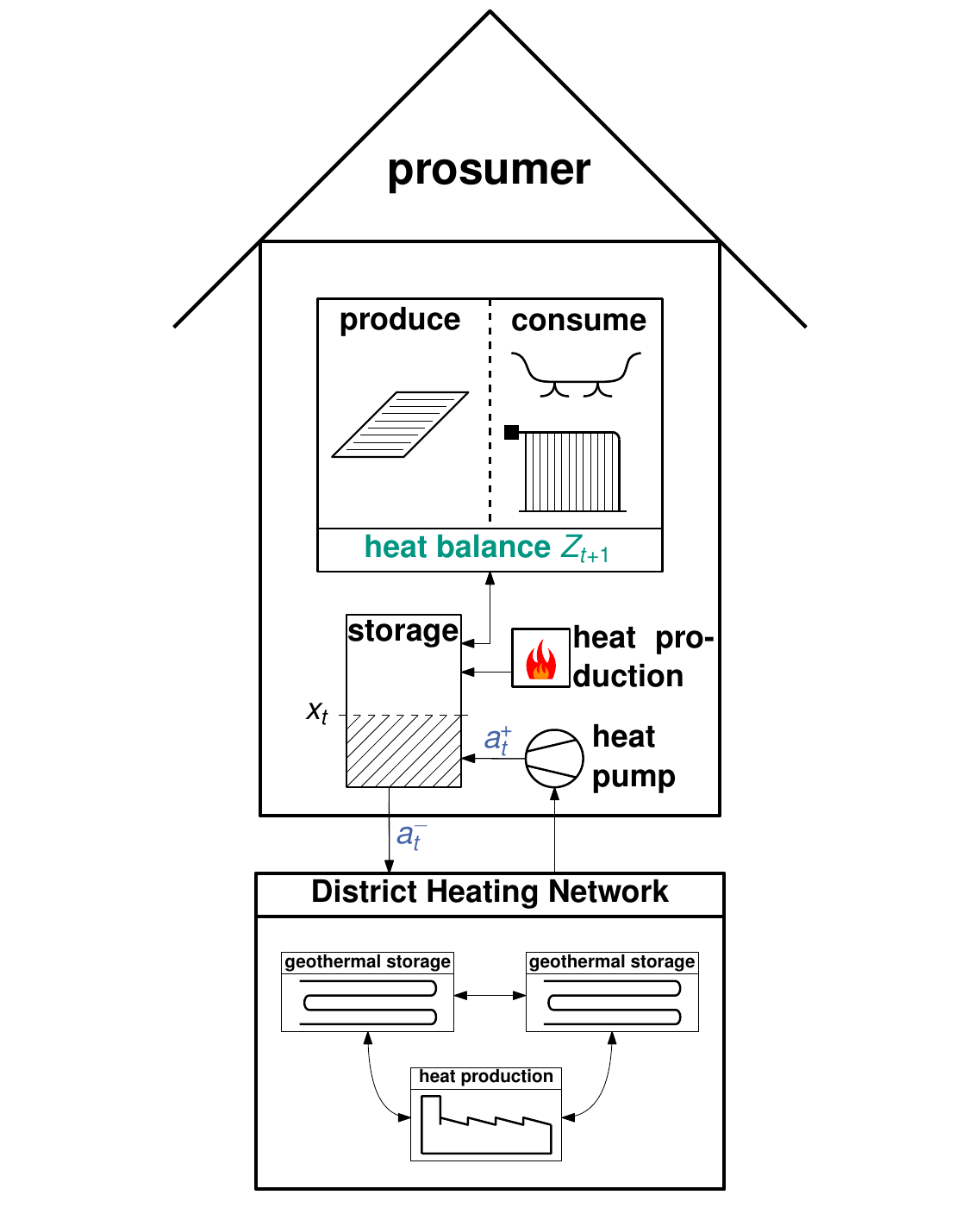}
	    \caption{Sketch of the Model}
        \label{fig:prosumer_model}
    \end{minipage} \hfill
    \begin{minipage}[t]{0.6\textwidth}
        \centering
        \includegraphics[width=\textwidth]{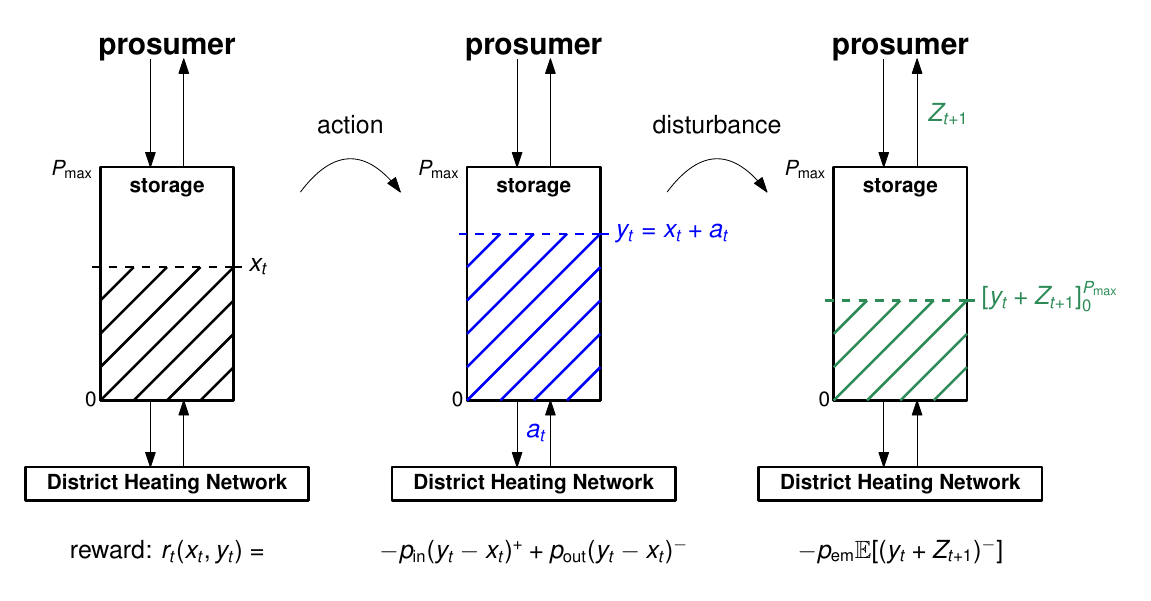}
    	\caption{Sketch of the Transition Function}
        \label{fig:state_transition}
    \end{minipage}
\end{figure}
First, we fix some notation. Let $a,b,x \in \R$ with $a \le b$. We define the positive and negative parts part by 
\[ x^+ \coloneqq \max\{x,0\} \qquad \text{and} \qquad x^-\coloneqq \max\{-x,0\} \]
and the clipping-function by
\[ [x]^b_a \coloneqq \min\{b,\max\{a,x\}\}.\]

We define the heat storage problem as an MDP as follows: Let $T\in \N$ be the time horizon and denote two sets of relevant time indices by $\timeindex \coloneqq \{0,1,\dots,T\}$ and $\timeindex' \coloneqq \{0,1,\dots,T-1\}$. The state of the MDP will be the state of charge of the hot water tank. Therefore, the state space is given by  $\statespace \coloneqq [0,\PufferStoremax]$ with $\PufferStoremax > 0$ being the maximal storage capacity. From time $t \in \timeindex'$ to time $t+1$ the transition can be decomposed into two parts (see \cref{fig:state_transition} for an illustration). Let $x \in \statespace$ be the state of charge of the hot water tank at time $t$. Firstly, an amount of heat energy $a$ for the next time unit can be bought from or sold to the DHN. If $a > 0$, then the amount $a$ of heat energy is bought from the DHN and charged to the hot water tank; if $a < 0$ then the amount $\abs{a}$ is sold to the DHN and discharged from the hot water tank. This results in a new state of the system, the so-called post-state $ y = x + a \in \statespace$. Thus,  the (dis-)charging action is implicitly given by $y-x$. Secondly, the prosumer provides some heat balance $Z_{t+1}$ at time $t+1$. If $Z_{t+1} > 0$, then the prosumer has excess heat production of $Z_{t+1}$ that needs to be stored; if $Z_{t+1} < 0$, then the prosumer requires an amount $\abs{Z_{t+1}}$ of heat energy from the storage.
If the new storage level $y$ does not cover the required heat demand, i.e.\ $y + Z_{t+1} < 0$, an emergency heat generator must be used to provide the difference. If $y + Z_{t+1} > \PufferStoremax$ the overflowing heat energy will be lost to the environment. We assume that $(Z_t)$ are independent random variables, where $Z_t$ has cumulative distribution function (cdf) $F_{t}.$ Note that the random variables are not necessarily identically distributed. Indeed, the distribution of $Z_t$ typically depends on the time of day and shows some periodicity.  Since there is some time delay between the 'order' of $a$ and its 'delivery', it is crucial to add $a$ to the storage before the actual need $Z_{t+1}$ is known.  Therefore, a state transition from the post-state $y$ at time $t$ to the state at time $t+1$ is given by
	 \[ X_{t+1}= K_t(y,Z_{t+1}) \coloneqq [y + Z_{t+1}]^{\PufferStoremax}_0.\]
Each part of the transition also contributes to the one-stage reward $r_t$ at time $t$. If heat energy is bought from the DHN, it is done so at the fixed price of $\pin> 0$ per unit of heat energy, i.e.\ the reward for heat energy bought is $-\pin (y-x)^+$. Similarly, the selling price is $\pout > 0$ per unit of heat energy, so the reward is $+\pout(y-x)^-$. Hence, for the first part the system gains the reward  
\[ \ri{t}(x,y) \coloneqq -\pin (y-x)^+ +\pout (y-x)^-\qquad (x,y\in\statespace, t \in \timeindex').\]
In the second part the system has to pay for the emergency heat production at price $\pem$ per unit of heat. The shortfall is $(y+Z_{t+1})^-$. As seen below in \cref{def: value function} we want to maximize the expected rewards. Therefore, we can replace the shortfall by its expectation value and the system gains the (expected) reward 
\[ \rii{t}(y) \coloneqq -\pem \EE\big[(y + Z_{t+1})^-\big] \qquad (y\in\statespace, t \in \timeindex'). \]
The full one stage reward is given by $r_t(x,y) \coloneqq r^{(1)}_t(x,y) + r^{(2)}_t(y)$, $x,y \in \statespace, t \in \timeindex'$. $r^{(2)}$ does not depend on $x$ so we will drop the $x$ from the notation and just write $r^{(2)}(y) = r^{(2)}(x,y)$. We assume that 
 $$0 < \pout \le \pin <\pem.$$
The inequality $\pout \le \pin$ prevents arbitrage and $\pin < \pem $ enforces the use of the emergency heat generator only for shortfalls. We define $$\psum \coloneqq \pout + \pin + \pem.$$ 
We assume that we can sell all heat energy that is left over at time $T$ at price $\pout$. So the terminal reward function is given by $r_T(x)=V_T(x) =  \pout x $, $x \in \statespace$. A policy in our model is a sequence $\pol=(\pol_0,\ldots,\pol_{T-1})$ of Markovian decision rules $\pol_t : \statespace \to \statespace$ mapping the state $x$ to the post-state $y$. We denote the set of all policies by $\polspace$. By classical MDP theory (see e.g.\ \cite{lerma96,puterman2014markov,bauerle2011MarkovDecisionProcesses}) it is enough to consider Markovian decision rules, since history dependent decision rules would not increase the value. Thus, when we introduce a canonical probability space and the state process $(X_t)$ on it, the policy value functions $V^\pol_t\colon \statespace \to\R$, $t \in \timeindex, \pol \in \polspace$, are given by
\begin{equation} \label{def: value function}
     V^\pol_t(x) \coloneqq  \EE^\pol_{t,x} \Big[\sum_{k=t}^{T-1} r_k(X_k,\pol_k(X_k))+ V_T(X_T)\Big],
\end{equation}
where $V^\pol_t(x)$ denotes the expected reward of the system when we start at time $t$ with storage $x$ and then act according to policy $\pol$. Similarly, we can define the (optimal) value functions  $V_t\colon \statespace \to\R$, $t \in \timeindex$ by  
\begin{align}\label{prob:MDP}
    V_t(x) = \sup_{\pol \in \polspace} V^\pol_t(x),
\end{align}
where $V_t(x)$ denotes the maximal expected reward of the system when we start at time $t$ with storage $x$. The  value functions  and the optimal state-post-state functions $Q_t\colon \statespace^2 \to \R$, $t \in \TimeIndex'$, satisfy the Bellman optimality equation recursively with
\begin{align*}
    V_T(x) &= \pout x,  && (x \in \statespace)\\
    Q_t(x,y) &=   r_t(x,y) + \EE\big[ V_{t+1}\big([Z_{t+1} + y]^{\PufferStoremax}_0 \big)\big] && (x,y \in \statespace, t \in \TimeIndex')\\
    V_t(x) &= \sup_{y \in [0,\PufferStoremax]}Q_t(x,y) && (x \in \statespace, t \in \TimeIndex').
\end{align*} 
Moreover, we define the functions $\holdt \colon \R \to \R$, $t \in \TimeIndex$, by $h_T \equiv 0$ and
	\begin{align*}
		\holdt(y) &= r^{(2)}(y) + \Exv[V_{t+1}(K_t(y,Z_{t+1}))] \\
        &=- \pem \Exv\left[(Z_{t+1}+y)^-\right]  + \Exv\left[V_{t+1}\big([Z_{t+1}+y]^\PufferStoremax_0\big)\right] && (y \in \R, t \in \TimeIndex').
	\end{align*}
and the functions $\hcht,\hdst\colon \R \to \R$, $t \in \TimeIndex$, by 
	\begin{align*}
	    \hcht(y) \coloneqq \holdt(y) - \pin y \qquad \text{and} \qquad \hdst(y) \coloneqq \holdt(y) - \pout y, &&(y \in \R, t \in \TimeIndex).
    \end{align*}
The functions $\hcht$ and $\hdst$ capture the effect of the post-action state $y$ on the value. More precisely, we can rewrite the optimal state-post-state function $Q_t$ by
\begin{align}
    Q_t(x,y) &\coloneqq  \ri{t}(x,y) + \rii{t}(y) + \Exv\big[ V_{t+1}\big([Z_{t+1} + y]^{\PufferStoremax}_0 \big)\big] \nonumber\\*
    &= \begin{cases}
        \hdst(y) + \pout x, &y < x,\\
        \holdt(x), & x=y,\\
        \hcht(y) + \pin x, &y > x.
    \end{cases} &&(x,y\in\statespace) \label{def:Qt by hcht and hdst}
\end{align}
Since the state space \(\mathcal X=[0,\PufferStoremax]\) is compact and the one-stage reward and transition operator are continuous, the Bellman maximization is attained. In particular, the functions \(Q_t(x,\cdot)\) are continuous on \(\mathcal X\). The explicit maximizing selector constructed below (see \cref{eq: optimal threshold policy (abs cont)}) is continuous except at at most two points and hence Borel measurable.

We will consider the derivative of functions that are defined over a convex, but not necessarily open subset of $\R^d$. In such cases, we will consider the derivative relative to the domain:

\begin{definition}
    Let $D \subseteq \R^d$ be convex, $x_0 \in D$. The function $f\colon D \to \R$ is called differentiable relative to $D$ at $x_0$ if there exists a vector $g \in \R^d$ such that
    \[ \lim_{\substack{x \to x_0\\x \in D\setminus \{x_0\}}} \frac{\abs{f(x)- f(x_0) - \langle g, x-x_0 \rangle}}{\norm{x-x_0}} = 0. \]
\end{definition}
We will use the derivative relative to the domain without explicitly mentioning it and just write derivative.  

In the next section we show that the optimal policy is of threshold type.

\section{Optimal policy}\label{sec:optimal_policy}
Before we explain the structure of the optimal policy, let us introduce some notations. We will denote the Lebesgue-measure on $\R$ with $\lambda$. With \enquote{almost everywhere} (a.e.) we mean with respect to the Lebesgue-measure (if not stated otherwise). Let $J \subseteq \R$ be an interval with either  $J = [a,b]$ for $a < b$ or $J =\R$. With $\LC(J)$ we denote the set of integrable functions on $J$ with respect to the Lebesgue-measure. With $\LCLoc(J)$ we denote the locally integrable functions on $J$.

\begin{definition}\label{def:AC}
	Let $-\infty < a < b < \infty$ and $G\colon [a,b] \to \R$. 
	\begin{itemize}
		\item[a)] 	We call $G$ \defi{absolutely continuous (a.c.)} and write $G \in \AC[a,b]$ if for every $\epsilon > 0$ there exists a $\delta > 0$ such that for all $a \le \alpha_1 < \beta_1 \le \alpha_2 < \beta_2 \le \dots \le \alpha_n < \beta_n \le b$ with $\sum_{k=1}^{n}(\beta_k-\alpha_k) < \delta$ we have
		\[ \sum_{k=1}^{n} \abs{G(\beta_k) - G(\alpha_k)} < \epsilon. \]
		\item[b)] We call $G$ \defi{absolutely continuously differentiable (a.c.d.)} and write $G \in \ACD[a,b]$ if $G$ is continuously differentiable on $[a,b]$ 
        and $G'$ is a.c.\ on $[a,b]$, i.e. $G'\in \AC[a,b]$.  
	\end{itemize}
\end{definition}
We will be mostly interested in functions defined on $\R$. Therefore we expand the definitions in \cref{def:AC}  similarly to \cite[p. 329, Definition 4.18 ]{elstrodt2018MassUndIntegrationstheorie} by requiring that the respective property holds on all compact intervals. We then write $\AC(\R)$ and $\ACD(\R)$ respectively.

\subsection{Structure in the general case}
Next, we need some mild assumptions on the random variables (r.v.) $(Z_t)_t$:

\begin{assumptions} \label{asm: assumtion for optimal policy abs cont}
	We assume that for  all $t \in \{1,\dots,T\}$,%
	\begin{enumerate}[(I)]
		\item the r.v. $Z_1, \dots, Z_T$ are independent and $\Exv\abs{Z_t} < \infty$,
		\item $F_t$ is the d.f.\ of $Z_t$ and there exists a probability density function (pdf) $f_t\colon \R \to [0,\infty)$ with $\Pmass^{Z_t} = f_t \lambda$.
	\end{enumerate}
\end{assumptions}
\begin{remark} \leavevmode

        Assumption \eqref{asm: assumtion for optimal policy abs cont} $\mathrm{(II)}$ is equivalent to $F_t \in \AC(\R)$ for all $t \in \{1,\dots,T\}$ (see \cite[Theorem 31.7]{billingsley1995ProbabilityMeasure} and the Radon–Nikodym theorem e.g.\ \cite[p.422, Theorem 32.2]{billingsley1995ProbabilityMeasure}).  
	
\end{remark}

Now we are prepared to investigate the structure of the optimal heat storage policy. It turns out that it is a so-called threshold policy. The main part of this section is now the following observation. The proof is postponed to \cref{sec:appendixA}.

\begin{theorem} \label{thm: optimal policies are threshold policies (abs cont)}
Suppose  \cref{asm: assumtion for optimal policy abs cont} is valid. Then for all $t\in\TimeIndex$ we have $V_t \in \ACD[0,\PufferStoremax]$, $V_t$ is concave and $V_t'\in [\pout, \pin]$.  Moreover, an optimal policy $\pol^*$ for the MDP is given by
	\begin{equation} \label{eq: optimal threshold policy (abs cont)}
		\pol^*_t(x) = [x]^\thds{t}_\thch{t} = \begin{cases}
			\thds{t}, & x > \thds{t},\\
			x,& \thch{t} \le x \le \thds{t},\\
			\thch{t}, & x < \thch{t},
		\end{cases}
	\end{equation}
	for $t \in \TimeIndex'$, where  $(\thch{t})_{t \in \TimeIndex'}$ and $(\thds{t})_{t \in \TimeIndex'}$ are two sequences of thresholds with $0 \le \thch{t} \le \thds{t} \le \PufferStoremax$. More precisely, it holds that 
	\[ \thch{t} \coloneqq \clipstorage{\tthcht} \quad \text{and} \quad \thds{t} \coloneqq \clipstorage{\tthdst}, \]
	where $\tthch{t}$ is the smallest solution $x$ of 
	\begin{align}\label{eq:ac_charge threshold (abs cont)}  \tag{charge}
	\pin = \pem F_{t+1}(-x) + \Exv\big[ V_{t+1}'(Z_{t+1} + x) \1_{\{0 \le Z_{t+1} + x \le \PufferStoremax\}}  \big] 
	\end{align}
	and $\tthds{t}$ is the largest solution $x$ of
	\begin{align} \label{eq:ac_discharg threshold (abs cont)} \tag{discharge}
	\pout =  \pem F_{t+1}(-x) + \Exv\big[ V_{t+1}'(Z_{t+1} + x)  \1_{\{0 \le Z_{t+1} + x \le \PufferStoremax\} }\big]. 
	\end{align}
\end{theorem}

Thus, the optimal policy works as follows: If the charge $x$ of the hot water tank at time $t$ is above $\thds{t}$, energy is sold to the DHN to bring it down to $\thds{t}$. If $x$ is below $\thch{t}$, then the remaining energy must be bought to bring it up to $\thch{t}$. If $x$ is between the thresholds we do nothing.

In the proof we use the structure of the problem and analyze the effects of the post-action $y$ on the value of the problem, captured by the function $h_t$.

\subsection{Optimal policy under Lipschitz assumptions}
In what follows, we will consider some slightly stronger assumptions which are easily satisfied in applications and yield stronger structural properties of the value functions. This is later needed to show convergence of our algorithms.

\begin{definition}
	Let $a<b$, let $g\colon[a,b]\to\R$ be a function, and let $L>0$. We say $g$ is \defi{Lipschitz (continuous)} if 
	\[ \Lip(g) \coloneqq \sup_{\substack{x,y \in [a,b],\\ x \ne y}} \frac{\abs{g(x) - g(y)}}{\abs{x-y}} < \infty. \]
	In this case we call $\Lip(g)$ the \defi{Lipschitz-constant} of $g$.
\end{definition}

\begin{assumptions} \label{asm: assumtion for optimal policy L}
	We assume that for  all $t \in \{1,\dots,T\}$,%
	\begin{enumerate}[(I)]
		\item the r.v. $Z_1, \dots, Z_T$ are independent and $\Exv\abs{Z_t} < \infty$,
		\item $F_t$ is the d.f.\ of $Z_t$ and there exists a probability density function (pdf) $f_t\colon \R \to [0,\infty)$ with $\Pmass^{Z_t} = f_t \lambda$. Further, there exists $L_Z<\infty$ such that $f_t(z)\le L_Z$ for a.e.\ $z\in\R$ and \(t=1,\ldots,T\).
	\end{enumerate}
\end{assumptions}

The next theorem adds to the previous one a statement about the Lipschitz continuity of $V_t'$. It uses the slightly stronger  \cref{asm: assumtion for optimal policy L} which also requires the density of the heat balance r.v.\ to be bounded. The proof is similar to the proof of  \cref{thm: optimal policies are threshold policies (abs cont)} and is sketched in \cref{sec:appendixB}.

\begin{theorem} \label{rem: properties of the value function under Lipschitz assumption}
	Suppose that  \cref{asm: assumtion for optimal policy L} is satisfied. 
   Then the value function $V_t$ is continuously differentiable on $[0,\PufferStoremax]$ and the derivative $V_t'$ is Lipschitz continuous on $[0,\PufferStoremax]$ for all $t \in \{0,\dots, T\}$. 
    Moreover, for $t \in \TimeIndex'$  and $L_{t+1} \coloneqq \Lip(F_{t+1})\le L_Z$ we get 
	\begin{equation} \label{eq: Lipschitz constant of value function (abs cont)}
		\Lip(V'_t) \le  \pem L_{t+1}\le \pem L_Z.
	\end{equation}
\end{theorem}

\section{Approximate solution algorithm via policy gradients}\label{sec:algorithm}
The proofs of statements in this section  are postponed to  \cref{sec:appendixC}. For the theoretical results, we need the following assumptions:
\begin{assumptions}\label{ass:PG} Let $L_\rho, \alpha_0, L_Z, \alpha_Z, \mu_Z > 0$.
    \begin{enumerate}[(I)] 
		\item the r.v.\ $X_0, Z_1, \dots, Z_T$ are independent and $\EE\abs{X_0}, \EE\abs{Z_{1}}, \dots, \EE\abs{Z_T} < \infty$. 
        \item The restriction of the law of \(X_0\) to \((0,\PufferStoremax )\) admits a probability density function (pdf) $f_{X_0}$, which is bounded from above by \(L_\rho\), i.e.\ 
        \[ f_{X_0}(x) \le L_\rho, \qquad x \in \statespace\setminus\{0,\PufferStoremax\}.  \]
        Moreover, it holds $\PP(X_0=0)\ge\alpha_0$ and $\PP(X_0=\PufferStoremax)\ge\alpha_0$.
        \item For all $t \in \TimeIndex'$ the law of $Z_{t+1}$ admits a pdf $f_{t+1}$ on $\R$ which is bounded from above  by $L_Z$ on $\R$ and bounded from below by $\mu_Z$ on $[-\PufferStoremax,\PufferStoremax]$, i.e.\
        \[ f_{t+1}(y) \le L_Z \text{ for } y \in \R \text{ and } f_{t+1}(y) \ge \mu_Z \text{ for } y \in [-\PufferStoremax,\PufferStoremax]. \]
        Moreover, it holds $\Pmass(Z_{t+1} \le -\PufferStoremax) \ge \alpha_Z$ and $\Pmass(Z_{t+1} \ge \PufferStoremax) \ge \alpha_{Z}$ for all \(t \in \TimeIndex'\). 
	\end{enumerate}
\end{assumptions}

In what follows define the constants
\begin{equation}
\alpha \coloneqq \min\{\alpha_0,\alpha_Z\},
\qquad
\Lambda \coloneqq \max\{L_\rho,L_Z,1\},
\label{eq:alpha-Lambda}
\end{equation}
and
\begin{equation}
\kappa \coloneqq \pem \mu_Z >0.
\label{eq:kappa}
\end{equation}

Motivated by the structural result of the optimal policy, we introduce the parameter space
\[ \thsspace \coloneqq \bigtimes_{t \in \TimeIndex'} \thsspace_t \quad \text{with} \quad \thsspace_t \coloneqq \{ (\lt,\ut) \in \R^2 \mid 0 \le \lt \le \ut \le \PufferStoremax \} \]
and the space of parameterized threshold policies
\[ \Pi_\thsspace \coloneqq \{ \pol^\ths \in \polspace \mid \ths \in \thsspace\} \quad \text{with} \quad \pol_t^\ths(x) \coloneqq [x]^{\ut_t}_{\lt_t} \quad (t \in \TimeIndex').  \]
For a parameterized policy $\pol^\ths$, we denote the parameterized policy value function  by $$V^\ths_t \coloneqq V^{\pol^\ths}_t, \qquad t \in \TimeIndex,$$
and the parameterized $h$-functions $\holdt^\ths\colon \R \to \R$  by $\hpara{T}(y) = 0$ and, 
    \[ \hpara{t}(y) \coloneqq \rii{t}(y) + \EE[\Vpara{t+1}(\clipstorage{y+\Dist{t+1}})], \qquad t \in \TimeIndex'. \]
By \cref{thm: optimal policies are threshold policies (abs cont)} we know that an optimal policy is contained in $\Pi_\thsspace$, i.e.\  $V_t(x) = \sup_{\ths\in \thsspace} V^\ths_t(x)$ for all $x \in \statespace$. Thus, we only have to optimize over the thresholds $\thcht$ and $\thdst$. However, in a general case, this becomes quite difficult. In this section, we state a gradient ascent algorithm that, for a given number of realizations of the disturbance sequence $(Z_{t+1})$, approximates the optimal thresholds. At the end of the section we will show convergence in expectation for the approximated solution. We start by defining the function we want to maximize, the objective function $\lo$:
\begin{definition}
    Let $\initdist$ be a distribution on $\statespace$ and $X_0 \sim \initdist$ be a randomized initial state of the system. We define the objective function $\lo \colon \thsspace \to \R$ by 
    \[ \lo(\ths) \coloneqq \EE_{\rho}[V^\ths_0(X_0)]\quad \mbox{ and } \lo^* = \sup_{\eta\in\Theta} \lo(\eta).\]
\end{definition}
\begin{remark}
    Note that the randomization with the initial state does not change the optimal policy. Indeed, we have $V^\vartheta_0(x)\le V_0(x)$ and $V^{\vartheta^*}_0(x)= V_0(x).$ So,
    \begin{align*}
        \ell(\vartheta) \le \EE_{\rho}[V_0(X_0)]=\ell(\vartheta^*),
    \end{align*}
    and hence $\vartheta^*$ is a maximizer independent of the choice of $\rho.$
\end{remark}

We want to maximize the expected reward of the system, i.e.\ the function $\lo$, by searching along its gradient $\nabla \lo$. We derive the following formulae for the gradient of $\lo$:

\begin{lemma}\label{lem:derivative in expectation}
    Under \cref{ass:PG} it holds:
    \begin{itemize}
        \item[a)] For all $\ths \in \thsspace$ the functions $\hpara{t}$, $t \in \TimeIndex$, are differentiable on $\R$ with $(\hpara{T})'(y) = 0$ and 
            \begin{align}
                (\hpara{T-1})'(y) &= \pem \PP(y + Z_{T} < 0) + \pout\PP(0<y+Z_{T}<\PufferStoremax), \nonumber \\
                (\hpara{t})'(y) &=  \begin{aligned}[t]
                &\pem \PP(y + Z_{t+1} < 0)\\
               &+\pin\PP(0<y+Z_{t+1}<\uth_{t+1})\\
                &+\EE\!\left[
                  (\hpara{t+1})'(y+Z_{t+1})
                  \1_{\{\uth_{t+1}<y+Z_{t+1}<\oth_{t+1}\}}
                \right]\\
                &+\pout\PP(\oth_{t+1}<y+Z_{t+1}<\PufferStoremax), && y \in \R,\ t \le T-2.
             \end{aligned} \label{eq: aprime}
            \end{align}
            We use the notation 
            $$\apara{t} \coloneqq (\hpara{t})' \qquad  t \in \TimeIndex.$$
        \item[b)] $\lo$ is differentiable on $\thsspace$ with
        \[\nabla \lo (\ths) = \big( \partial_{\lt_0} \lo(\ths), \partial_{\ut_0} \lo(\ths), \dots, \partial_{\lt_{T-1}} \lo(\ths) , \partial_{\ut_{T-1}} \lo(\ths) \big), \qquad \ths \in \thsspace,  \] 
        where, for $t \in \TimeIndex'$,
        \begin{align}
            \partial_{\lt_t}\lo(\ths) &= \PP_\rho^\ths(X_t\le \lt_t,\ X_t\neq\Pmax) \ \big( \apara{t}(\lt_t)-\pin \big), \label{eq:grad-lower} \\
            \partial_{\ut_t} \lo(\ths) &= \PP_\rho^\vartheta(X_t\ge \ut_t,\ X_t\neq0) \  \big( \apara{t}(\ut_t)-\pout\big),\label{eq:grad-upper}
        \end{align}
        where $\PP^\ths_\rho(X_t \in \cdot)$ is the distribution of the state $X_t$ at time $t$ under the initial distribution $X_0 \sim \rho$ and the policy $\pi^\ths$.
    \end{itemize}
\end{lemma}

However, we cannot directly evaluate $\lo(\ths)$ or $\nabla \lo(\ths)$, due to the unknown distribution of the collected source of randomness in the system, given by the random vector $\Episode \coloneqq (X_0, Z_1, \dots, Z_T)$ with values in $\episodespace \coloneqq \statespace \times \R^T \subseteq \R^{T+1}$. Therefore, we are interested in the  value of the objective function and its gradient along a sample path, i.e.\ a realization $\episode \coloneqq (\initstate, \dist{1}, \dots, \dist{T}) $ of $\Episode$.  
\begin{definition} \label{def: realized value function}
    Let $\episode \in \episodespace$ be as sample path and $\ths \in \thsspace$ be a policy parameter. The realized  value functions $\Vsmp{t}\colon \statespace \to \R$, $t \in \TimeIndex$,  are defined recursively by 
    \begin{align*}
        \Vsmp{T}(x) &\coloneqq \pout x\\
        \Vsmp{t}(x) &\coloneqq r_t^{(1)}\big(x, \pol^\ths_t(x)\big) + \riismp{t}\big( \pol^\ths_t(x)\big) + \Vsample^{\ths, \episode}_{t+1}\big(\clipstorage{\smash{\pol^\ths_t(x) + \dist{t+1}}} \big),   &&t \in \TimeIndex',
    \end{align*}
    where $\riismp{t}(y) \coloneqq -\pem (y + \dist{t+1})^-$ is the realized second part of the rewards of the system at time $t$ and in state $x\in \statespace$. Similarly, $\hsmp{t}\colon \R \to \R$, $t \in \TimeIndex$, are given by
    \begin{align}
        \hsmp{T}(y) &\coloneqq 0,\nonumber \\
        \hsmp{t}(y) &\coloneqq \riismp{t}(y) + \Vsmp{t+1}(\clipstorage{y+\dist{t+1}}),  &&t \in \TimeIndex'. \label{eq: half recursion step realized h}
    \end{align}
    Then the realized objective function $\lsmp\colon \thsspace \to \R$ is given by 
    \[\lsample^{\episode}(\ths) \coloneqq \Vsample^{\ths, \episode}_0(\initstate).\]
\end{definition}

\begin{remark}
    When taking the expectation value with respect to $\Episode$ of the sample path functions, we get the parameterized functions defined at the start of the section, i.e.\ 
    \[ \Vpara{t}(x) = \EE_\Episode[\Vrv{t}(x)], \quad \hpara{t}(y) = \EE_\Episode[\hrv{t} (y)] \quad \text{and} \quad \lo(\ths) = \EE_\Episode[\lrv(\ths)] \]
for every $x \in \statespace$, $y \in \R$, $\ths \in \thsspace$ and $t \in \TimeIndex$.
\end{remark}

The derivatives of $\hsmp{t}$ and $\lsmp$ do not exist for every sample path  $\episode \in \episodespace$ and every point in their respective domains due to some kink points in the functions. However, under the given assumptions there are sufficiently few kink points in the following sense:

\begin{lemma}\label{lem:pathhderivative} 
	Suppose that  \cref{ass:PG} is satisfied. 
We obtain the following pathwise derivative recursion and stochastic policy-gradient estimator:
\begin{itemize}
    \item[a)] For  $\episode \in \episodespace$ let $\asmp{t}\colon \R \to \R$, $t \in \TimeIndex$, be defined by backward recursion as $\asmp{T}(y) = 0$,
    \begin{align}
        \asmp{T-1}(y)
        &\coloneqq\pem\1_{\{y+z_{T}<0\}}+\pout\1_{\{0<y+z_{T}<\PufferStoremax\}},
        \nonumber\\
        \asmp{t}(y) &\coloneqq
        \begin{aligned}[t]
            &\pem\1_{\{y+z_{t+1}<0\}}+ 
            \pin\1_{\{0<y+z_{t+1}<\uth_{t+1}\}}\\
            &+ \asmp{t+1}(y+z_{t+1})
            \1_{\{\uth_{t+1}<y+z_{t+1}<\oth_{t+1}\}}\\
            &+\pout\1_{\{\oth_{t+1}<y+z_{t+1}<\Pmax\}}  &\text{for} \ t\le T-2 .
        \end{aligned} 
        \label{eq:pathwise-a}
    \end{align}
    For all $t \in \TimeIndex$, $y \in \R$ and $\Pmass^\Episode$-almost all $\episode \in \episodespace$ the derivative $(\hsample^{\ths,\episode}_{t})'(y)$ exists and  
    $$(\hsample^{\ths,\episode}_{t})'(y) = \asmp{t}(y).$$
    \item[b)]  Define the function $G\colon \thsspace \times \episodespace \to \R^{2T}$ with 
       $$G(\vartheta,e)\coloneqq \big(G_{0,-}(\vartheta,e),G_{0,+}(\vartheta,e),\ldots , G_{T-1,-}(\vartheta,e),G_{T-1,+}(\vartheta,e)\big),$$
    where, for $t \in \TimeIndex'$,
    \begin{align}
         G_{t,-}(\vartheta,e)
         &\coloneqq\1_{\{\xsmp{t}\le\uth_t,\ \xsmp{t}\ne\PufferStoremax\}}
               \bigl(\asmp{t}(\uth_t)-\pin\bigr),
         \label{eq:C4}\\
         G_{t,+}(\vartheta,e)
         &\coloneqq\1_{\{\xsmp{t}\ge\oth_t,\ \xsmp{t}\ne0\}}
               \bigl(\asmp{t}(\oth_t)-\pout\bigr),
         \label{eq:C5}
    \end{align}
    where \((\xsmp{t})_{t=0}^T\) is the sample state path generated by \(e\), that is,
    \[\xsmp{0}\coloneqq x_0,\qquad \xsmp{t+1} \coloneqq
     \left[
        [\xsmp{t}]_{\uth_{t}}^{\oth_{t}}+z_{t+1}
     \right]_0^{\PufferStoremax},\qquad t\in \TimeIndex'.
    \]
    For all \(\vartheta\in\Theta\) and $\Pmass^\Episode$-almost all $e \in \episodespace$ the gradient $\nabla \lsmp(\ths)$ exists and
    $$ G(\ths,\episode) = \nabla \lsmp(\ths). $$

\end{itemize}
\end{lemma}

The null sets $\episodespace_0 \subseteq \episodespace$, where $(\hsmp{t})'(y)$ does not exist or does not coincide with $\asmp{t}(y)$ will depend on the argument $y$, and the same holds for $\nabla \lsmp(\ths)$. 
\begin{remark} \label{rmk: switching expectation value and derivatie}
    The following relationship holds for all $y \in \R$, $\ths \in \thsspace$ and $t \in \TimeIndex'$:
    \begin{align}
        \nabla\ell(\ths) = \EE_\Episode[G(\ths,\Episode)] \quad \text{and} \quad \apara{t}(y) = \EE_\Episode[\asample^{\ths,\Episode}_{t}(y)].  \label{eq: relation parameterized and realized functions}
    \end{align}
\end{remark}

For a gradient ascent step we need an estimator of the gradient of $\lo$. sample path derivatives motivate the following mean stochastic-gradient estimator:

\begin{definition}\label{def:samplegradient}
    Let $\ths \in \thsspace$, $N \in \N$ and $\Episode^{(1)}, \dots, \Episode^{(N)}$ be i.i.d.\  copies of $\Episode$. Then the sample gradient is defined by
    \[  \gradientsmp \lo_N(\ths) \coloneqq \frac{1}{N} \sum_{n=1}^N G(\ths,{\Episode^{(n)}}). \]
    
\end{definition}

Finally we need to ensure that the parameter still resides in the parameter space $\thsspace$. Therefore we apply a Euclidean projection $\proj_{\thsspace}$ onto $\thsspace$. In total, this leads to the projected stochastic policy-gradient (PG) algorithm:
\begin{algorithm}[H]
	\caption{projected stochastic PG}
	\label{alg:Simultaneous Policy Gradient}
	\begin{algorithmic}
        \Require Initial parameter $\ths^{(0)} \in \thsspace$, batch size $N$, iterations $K$
        \State Set the step size at $\eta = \frac{1}{L_{\nabla \lo}}$
        \For{$k \in \{0,\dots,K-1\}$}
            \State Draw a fresh batch of $N$ realizations $\episode^{(k,1)}, \dots, \episode^{(k,N)}$ of $\Episode = (X_0, Z_1,\dots,Z_T)$
            \State Calculate the sample gradient: $\gradientsmp \lo_N(\ths^{(k)}) = \frac{1}{N} \sum_{n=1}^N G(\ths^{(k)},{\episode^{(k,n)}})$
			\State Update policy parameter: $\ths^{(k+1)} = \proj_{\thsspace} \Big( \ths^{(k)} + \eta \gradientsmp \lo_N(\ths^{(k)}) \Big)$
		\EndFor \\
		\Return $\ths^{(K)}$
	\end{algorithmic}
\end{algorithm}

In the remaining part of this section we will use the framework provided in \cite{chen2026LandscapePolicyOptimization} and derive a convergence result. In \cite{chen2026LandscapePolicyOptimization} the authors showed that the Polyak-\L ojasiewicz-Kurdyka  (\PLK) condition together with $L$-smoothness of the objective function is sufficient for convergence of the projected stochastic PG algorithm. We define the two properties here in the maximization version (see \cite[Definition 1 and Assumption 1]{chen2026LandscapePolicyOptimization}):

\begin{definition}
    Let $d \in \N$, $D \subseteq \R^d$ convex and $g \colon D \to \R$ differentiable with $g^* \coloneqq \sup_{x \in D} g(x) < \infty $.
    \begin{enumerate}
        \item $g$ fulfills the \PLK condition if there exists a $\mu > 0$ such that 
        \[ g^* - g(x) \le \frac{1}{2\mu} \min_{n \in N_D(x)} \norm{\nabla g(x) - n}^2, \qquad x \in D,\]
        where $N_D(x)\coloneqq\{n\in\R^d:\ \langle n, y-x \rangle\le 0\text{ for all }y\in D\}$ is the normal cone of $D$ at $x$. $\mu$ is called the \PLK constant. 
        \item $g$ is called $L$-smooth if there is an $L> 0$ such that $\nabla g$ is $L$-Lipschitz continuous, i.e.\ 
        \[ \norm{\nabla g(x) - \nabla g(y)} \le L\norm{x-y}, \qquad x,y \in D. \]
    \end{enumerate}
\end{definition}

Moreover, in \cite[Theorem 1]{chen2026LandscapePolicyOptimization} they provided a way of checking the \PLK condition for the objective function based on properties of the underlying Markov decision process. A version of this theorem that is adapted to the notation of this work can be found in the appendix (see \cref{thm: Theorem cited from Landscape Paper}). In Section 7 of \cite{chen2026LandscapePolicyOptimization} the authors showed that the \PLK condition holds for the related cash balancing problem.  In addition to the \PLK condition, the objective function $\lo$ needs to be $L$-smooth and the gradient estimator $\gradientsmp \lo_N$ needs to be unbiased and bounded in variance.


Let \(\mathcal F_k\) be the sigma-field generated by the initial parameter and all batches
used before iteration \(k\).  If the batch
\(E^{(k,1)},\dots,E^{(k,N)}\) used at iteration \(k\) is independent of
\(\mathcal F_k\), then \(\vartheta^{(k)}\) is \(\mathcal F_k\)-measurable and we obtain:

\begin{lemma}[Unbiased and bounded variance]
\label{lem:unbiased}
	Suppose that  \cref{ass:PG} is satisfied. For every \(\vartheta^{(k)}\in\ThetaSet\),
\[
\E[\widehat\nabla_N\ell(\vartheta^{(k)}) | \mathcal{F}_k]
=
\nabla\ell(\vartheta^{(k)}).
\]
Moreover, with 
\begin{equation}
    \sigma^2\coloneqq2\psum^2T,  \label{eq:gradient-variance}
\end{equation}
it holds 
\begin{equation}
\E\left[ \norm{ \widehat\nabla_N\ell(\vartheta^{(k)}) - \nabla\ell(\vartheta^{(k)})}_2^2| \mathcal{F}_k \right] \le \frac{\sigma^2}{N}.
\end{equation}
\end{lemma}

Next we can show the L-smoothness of the objective function $\ell$.

\begin{lemma}[$L_{\nabla \lo}$-smoothness]
\label{lem:smooth}
Under \cref{ass:PG}, the function \(\ell\) is $L_{\nabla \lo}$-smooth on \(\ThetaSet\) with
\begin{equation}
L_{\nabla\ell} \coloneqq  4\psum\Lambda T^2.
\label{eq:L-gradient}
\end{equation}
Thus
\[
\norm{
\nabla\ell(\vartheta)
-
\nabla\ell(\vartheta')
}_2
\le
L_{\nabla\ell}
\norm{\vartheta-\vartheta'}_2
\]
for all \(\vartheta,\vartheta'\in\ThetaSet\).
\end{lemma}

Last but not least we show that a \PLK inequality is satisfied in our model. 
\begin{lemma}[\PLK inequality]
\label{lem:KL}
Under \cref{ass:PG}, the objective \(\ell\) satisfies a \PLK condition, i.e. 
\begin{equation}
\ell^*-\ell(\ths) \le \frac{1}{2\mu_\ell} \min_{n \in N_\thsspace(\ths)} \norm{ \nabla \lo(\ths) - n}^2, \qquad \vartheta\in\ThetaSet, \label{eq:KL-global}
\end{equation}
and one may take
\begin{equation}
\mu_\ell \coloneqq \frac{\kappa^3\alpha^8}{16eL_Z^2\psum^2T^2}.
\label{eq:mu-ell}
\end{equation}
\end{lemma}

The following theorem gives a global finite-sample bound on the expected optimality gap.  Thus, after a sufficient number of steps $K$ and using enough samples $N$ we obtain an approximate optimal policy. The bounds in the theorem indicate how large $K$ and $N$ must be. However, they depend strongly on the explicit model constants and may be very conservative. With  \cref{lem:unbiased,lem:smooth,lem:KL} the theorem follows from \cite[Lemma 1]{chen2026LandscapePolicyOptimization}.

\begin{theorem}
\label{thm:PG}
Suppose \cref{ass:PG} holds. Let $\sigma^2$, $L_{\nabla \lo}$ and $\mu_\lo$ be as in \cref{eq:gradient-variance,eq:L-gradient,eq:mu-ell}. Further, let the initial optimality gap be given by 
\[ \Delta_0 \coloneqq \ell^*-\ell(\vartheta^{(0)}). \]
Then the iterates generated by the projected stochastic PG-Algorithm (see \cref{alg:Simultaneous Policy Gradient}) satisfy
\begin{align}
\E\left[ \ell^*-\ell(\vartheta^{(K)}) \right] &\le \left( 1-\frac{\mu_\ell}{16L_{\nabla \lo}+\mu_\ell} \right)^K \Delta_0 + \frac{17\sigma^2}{\mu_\ell N}.
\label{eq:expected-convergence}
\end{align}

In particular, for
$\delta>0$, $q\in(0,1)$, if $\Delta_0>0$, 
\begin{equation}
N \ge \frac{34\sigma^2}{\mu_\ell\,\delta q} = \frac{68\psum^2T}{\mu_\ell\,\delta q}, \quad \text{and} \quad K \ge \frac{16L_{\nabla\lo}+\mu_\ell}{\mu_\ell}\log \left(\frac{2\Delta_0}{\delta q}\right),
\label{eq:N-K-bound}
\end{equation}
then
\begin{equation}
\PP\left(\ell^*-\ell(\vartheta^{(K)})>\delta\right)\le q.
\label{eq:high-probability}
\end{equation}
\end{theorem}

The same argument as in the previous theorem also yields a finite-sample mean-square error bound for the threshold parameters.

\begin{corollary}\label{cor:threshold_bounds}
    Suppose \cref{ass:PG} holds. Then we obtain with the notations of the previous theorem that
    \begin{align}\label{eq:threshold_difference}
        \E\| \vartheta^{(K)} -\vartheta^*\|_2^2 \le \frac{2}{\alpha \kappa  } \Big( \left(1-\frac{\mu_\ell}{16L_{\nabla \lo}+\mu_\ell} \right)^K\Delta_0+\frac{17\sigma^2}{\mu_\ell N}\Big).
    \end{align}
\end{corollary}


\section{Numerical Results}\label{sec:numerics}
We implemented the projected stochastic policy-gradient algorithm for a \(24\)-hour horizon, following the pathwise derivative recursion \cref{eq:pathwise-a}, the estimator \(G\) in \cref{eq:C4,eq:C5}, and the projected update in \cref{alg:Simultaneous Policy Gradient}.

For the illustration we used
$$ P_{\max}=1,\qquad p_{\rm out}=0.4,\quad p_{\rm in}=1,\quad p_{\rm em}=3, $$
and independent $Z_t \sim N(\mu_t,\sigma_Z^2)$ with $\mu_t =0.28 \sin\Big( \frac{2\pi(t-6)}{24}\Big)-0.02$ and $\sigma_Z = 0.22$. Thus, the system has a heat deficit overnight/morning, a positive surplus around midday, and again a deficit in the evening, see \cref{fig:heat_balance}.

\begin{figure}[H]
	\includegraphics[scale=0.6]{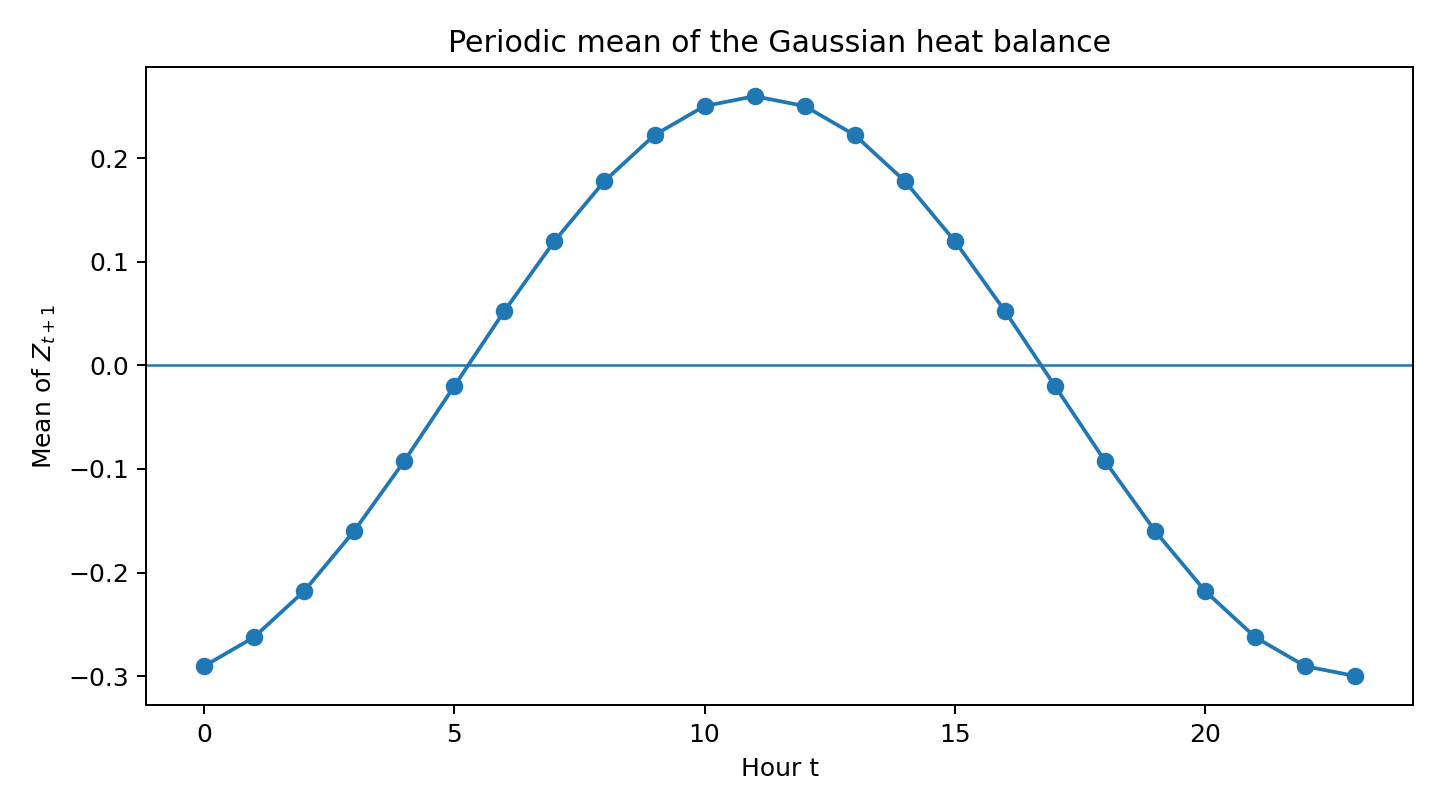}
	\caption{Mean of the heat balance.}\label{fig:heat_balance}
\end{figure}

For the  initial-state distribution $\rho$ we chose a probability \(0.1\) at each endpoint and the remaining \(0.8\) uniformly distributed on \((0,1)\) to meet the assumptions of  \cref{thm:PG}. The initial threshold profile is constant with
$$ \underline\vartheta_t^{(0)}=0.25,\qquad \overline\vartheta_t^{(0)}=0.75, \qquad t \in \TimeIndex'.$$
The baseline experiment uses
\[
N=2048,\qquad K=1200,\qquad \eta=0.01.
\] 
The training budget is therefore $KN=2\,457\,600$ episodes, containing $TKN=58\,982\,400$ disturbance realizations. Evaluation episodes are additional and are not included in this training budget. For this example, \cref{eq:L-gradient} gives
\[
L_{\nabla\lo}\approx18383.26,
\qquad
1/L_{\nabla\lo}\approx5.44\times10^{-5}.
\]
This is an extremely conservative step size, and we use the empirical step size $\eta=0.01$ instead. We received the following results for the thresholds listed in  \cref{tab:thresholds_K1200}.  For time points $t \in \{0,6,12,18,23\}$  \cref{fig:selected_thresholds} plots the upper and lower thresholds over the $K=1200$ iterations. The threshold policy $(\ths^{(k)}_t)_t$ is plotted in \cref{fig:threshold_profile} for selected $k \in\{0,25,100,300,600,1200\} $. The desired storage level is high before periods of expected deficit and much lower around the period of positive heat balance. Near the end of the horizon there is also a visible terminal-value effect. Using $20,000$ samples we estimate the objective function $\lo$, see \cref{fig:objective_pr}. 

\begin{figure}
	\includegraphics[scale=0.7]{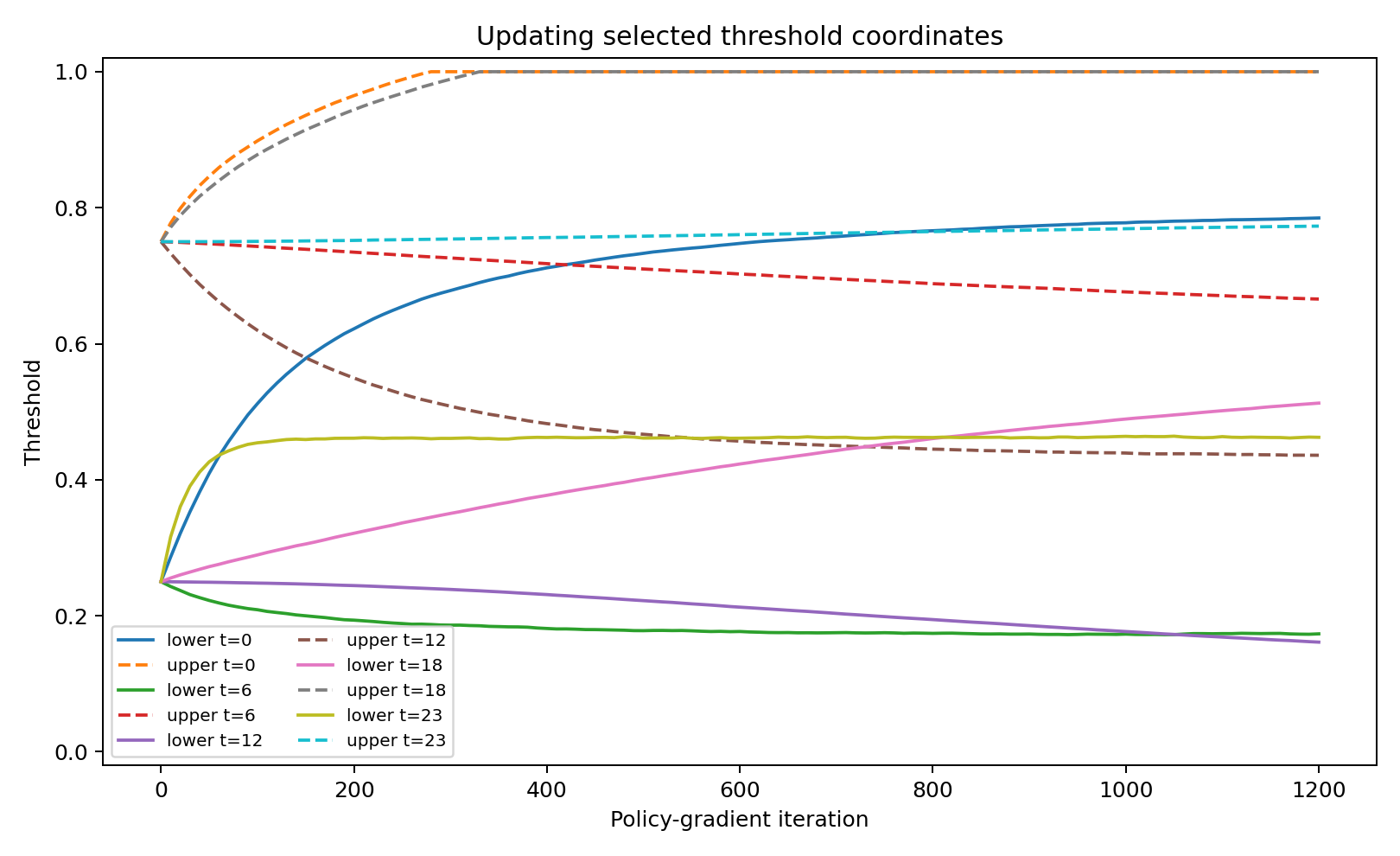}
	\caption{Plots of threshold iterations $(\ths_t^{(k)})_k$ for $t \in \{0,6,12,18,23\}$}\label{fig:selected_thresholds}
\end{figure}

\newpage
\begin{table}[htbp]
\centering
\caption{Threshold parameters from one training run after $K=1200$ iterations with batch size $N=2048$ and step size $\eta=0.01$. The disturbances satisfy $Z_{t+1}\sim N(\mu_{t+1},\sigma_Z^2)$ with $\sigma_Z=0.22$.}
\label{tab:thresholds_K1200}
\begin{tabular}{c|r|cc}
\hline
$t$ & $\mu_{t+1}$ 
    & $\underline{\vartheta}_t^{(1200)}$
    & $\overline{\vartheta}_t^{(1200)}$ \\
\hline
0  & $-0.2905$ & $0.7849$ & $1.0000$ \\
1  & $-0.2625$ & $0.7220$ & $0.9980$ \\
2  & $-0.2180$ & $0.6221$ & $0.9247$ \\
3  & $-0.1600$ & $0.5090$ & $0.8619$ \\
4  & $-0.0925$ & $0.3901$ & $0.7979$ \\
5  & $-0.0200$ & $0.2753$ & $0.7362$ \\
6  & $ 0.0525$ & $0.1733$ & $0.6657$ \\
7  & $ 0.1200$ & $0.0888$ & $0.5721$ \\
8  & $ 0.1780$ & $0.0481$ & $0.4779$ \\
9  & $ 0.2225$ & $0.0646$ & $0.4180$ \\
10 & $ 0.2505$ & $0.1109$ & $0.3894$ \\
11 & $ 0.2600$ & $0.1419$ & $0.3923$ \\
12 & $ 0.2505$ & $0.1611$ & $0.4359$ \\
13 & $ 0.2225$ & $0.1872$ & $0.5279$ \\
14 & $ 0.1780$ & $0.2221$ & $0.6469$ \\
15 & $ 0.1200$ & $0.2584$ & $0.7801$ \\
16 & $ 0.0525$ & $0.2979$ & $0.9108$ \\
17 & $-0.0200$ & $0.3722$ & $1.0000$ \\
18 & $-0.0925$ & $0.5126$ & $1.0000$ \\
19 & $-0.1600$ & $0.6423$ & $1.0000$ \\
20 & $-0.2180$ & $0.7101$ & $1.0000$ \\
21 & $-0.2625$ & $0.7183$ & $1.0000$ \\
22 & $-0.2905$ & $0.6482$ & $0.9169$ \\
23 & $-0.3000$ & $0.4623$ & $0.7729$ \\
\hline
\end{tabular}
\end{table}

\newpage

\begin{figure}[H]
	\includegraphics[scale=0.7]{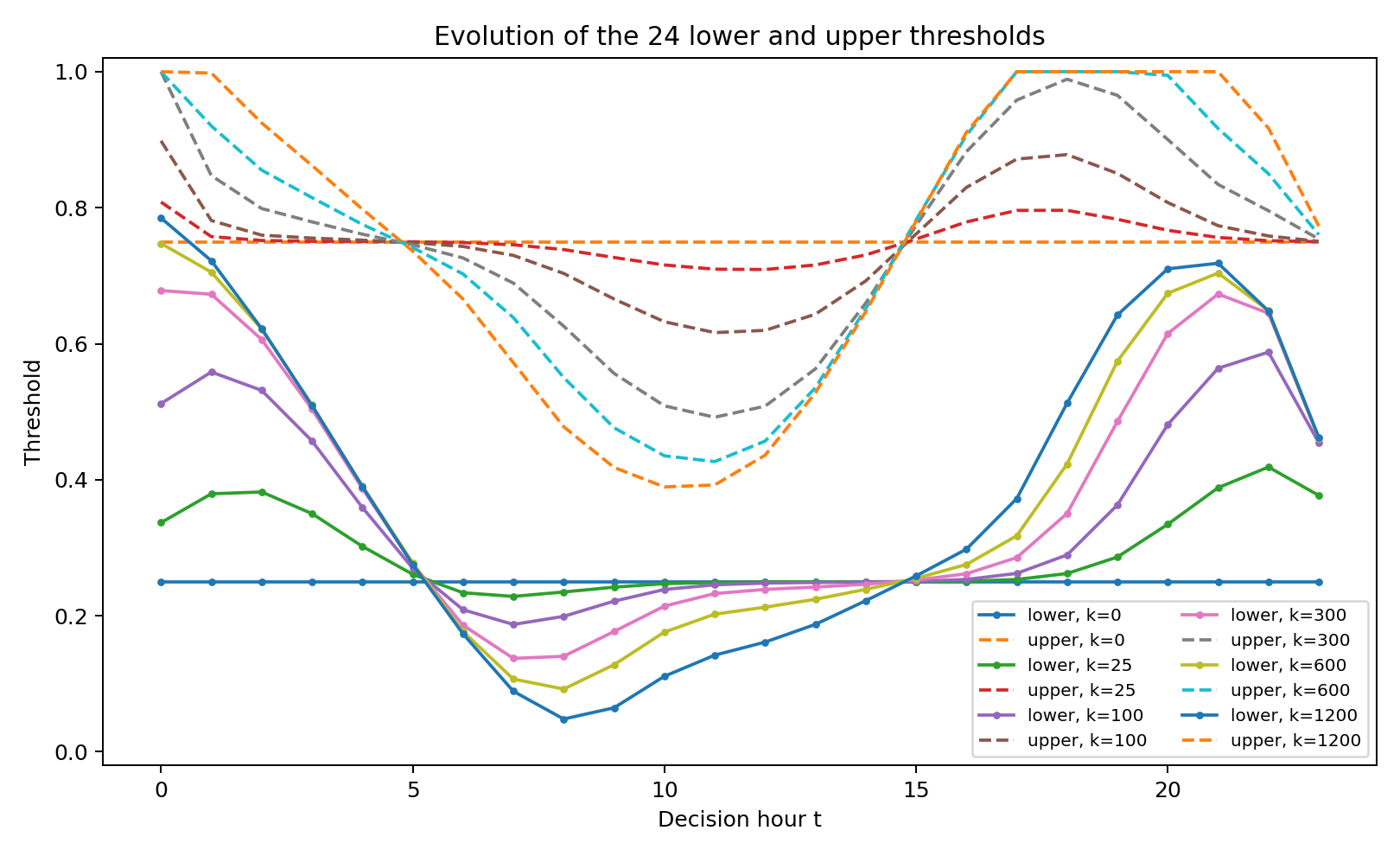}
	\caption{Plots of threshold policy $(\ths^{(k)}_t)_t$ for $k \in\{0,25,100,300,600,1200\}$}\label{fig:threshold_profile}
\end{figure}

\begin{figure}[H]
	\includegraphics[scale=0.7]{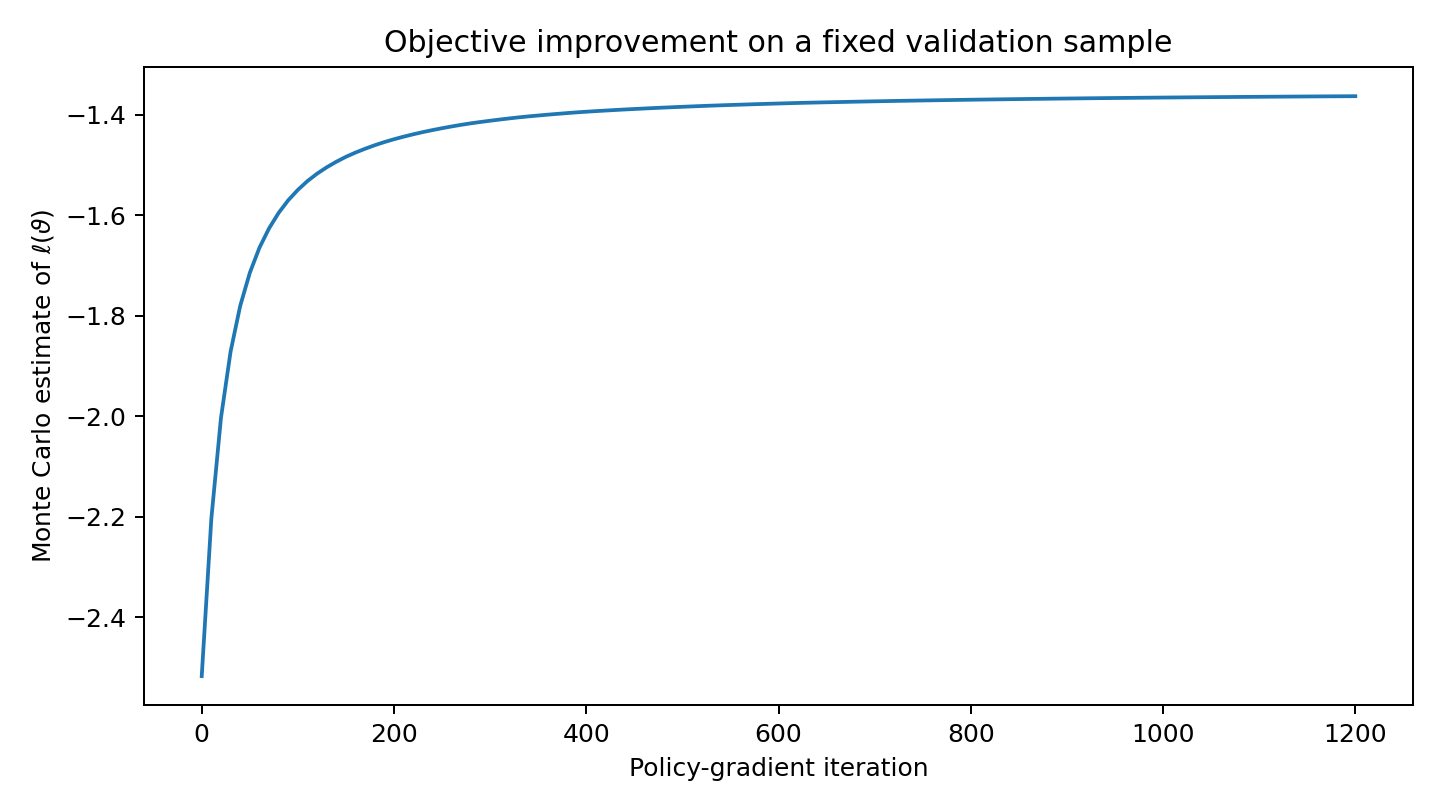}
	\caption{Monte Carlo estimates of $\ell(\vartheta).$}\label{fig:objective_pr}
\end{figure}

\newpage
\appendix 

\section{ Proof of  \cref{thm: optimal policies are threshold policies (abs cont)}}\label{sec:appendixA}

Before we start, we need the following auxiliary result:

\begin{lemma}
\label{lem:clipped-expectation}
Let $Z$ be an integrable random variable with absolutely continuous
distribution function $F$ and density $f$. Let
$G\in ACD[0,P_{\max}]$ and define
\[
H(y)
\coloneqq
\E\left[
G\left([Z+y]_0^{P_{\max}}\right)
\right],
\qquad y\in\R.
\]
Then $H\in ACD(\R)$ and
\begin{align}
H'(y)
&=
\E\left[
G'(Z+y)
\1_{\{0<Z+y<P_{\max}\}}
\right]
\label{eq:Hprime-simple}\\
&=
G'(P_{\max})F(P_{\max}-y)
-
G'(0)F(-y)
-
\int_0^{P_{\max}}G''(z)F(z-y)\,dz.
\label{eq:Hprime-ibp}
\end{align}
Moreover, for almost every $y\in\R$,
\begin{equation}
H''(y)
=
G'(0)f(-y)
-
G'(P_{\max})f(P_{\max}-y)
+\int_0^{P_{\max}}G''(z)f(z-y)\,dz.
\label{eq:Hsecond}
\end{equation}
\end{lemma}

\begin{proof}
Since $G$ is continuous on the compact interval
$[0,P_{\max}]$, the random variable
$G([Z+y]_0^{P_{\max}})$ is bounded and hence integrable. We can rewrite $H$ as
\begin{align*}
    H(y)&= G(0)F(-y) +G(\PufferStoremax)(1-F(\PufferStoremax-y))+\int_0^\PufferStoremax G(z)f(z-y)dz.
\end{align*}

Using integration by parts, we obtain
\begin{equation}
H(y)
=
G(P_{\max})
-
\int_0^{P_{\max}}
G'(z)F(z-y)\,dz.
\label{eq:H-representation}
\end{equation}
Since for $z,y_1,y_2\in\R$
\begin{align*}
    F(z-y_1)-F(z-y_2)=\int_{y_1}^{y_2} f(z-s)ds,
\end{align*}
we obtain with Fubini
\begin{align*}
    H(y_2)-H(y_1)&= \int_{y_1}^{y_2} \left(\int_0^{P_{\max}}
G'(z)f(z-s)\,dz. \right)\,ds
\end{align*}
Consequently $H$ is absolutely continuous and 
\[
H'(y)
=
\int_0^{P_{\max}}
G'(z)f(z-y)\,dz,\quad \mbox{ for a.e.\ } y,
\]
which is exactly \eqref{eq:Hprime-simple} after the substitution
$z=Z+y$. Now extend $G'$ by zero outside of $[0,P_{\max}].$ Since the extension belongs to $L^1(\R)$ and $f\in L^1(\R)$ the right-hand side is a convolution of two $L^1$- functions. It therefore has a continuous representative. Hence the above identity actually determines a continuous derivative of $H.$

Since $G'\in AC[0,P_{\max}]$, another integration by parts gives
\[
\int_0^{P_{\max}}G'(z)f(z-y)\,dz
=
G'(P_{\max})F(P_{\max}-y)
-
G'(0)F(-y)
-
\int_0^{P_{\max}}G''(z)F(z-y)\,dz.
\]
This proves \eqref{eq:Hprime-ibp}.

Finally, $G''\in L^1[0,P_{\max}]$. Hence, by Fubini's theorem
and the absolute continuity of $F$, the last expression is locally
absolutely continuous as a function of $y$. Differentiating it
almost everywhere yields \eqref{eq:Hsecond}. Thus
$H'\in AC_{\mathrm{loc}}(\R)$ and therefore $H\in ACD(\R)$.
\end{proof}

\begin{proof}[Proof of  \cref{thm: optimal policies are threshold policies (abs cont)} by backwards induction.]
	$V_T\in\ACD[0,\PufferStoremax]$ is concave, and $V_T'=\pout\in[\pout,\pin]$.
    
    For our induction hypothesis, we assume that the assertion holds for $V_{t+1}.$
    
    We  consider first the function $\holdt$. First, define 
    \begin{align*}
        g_t(y)\coloneqq -\pem \E[(Z_{t+1}+y)^-].
    \end{align*}
    By the assumption $\E|Z_{t+1}|<\infty$ we obtain a.e.\
    \begin{align*}
        g_t'(y) = \pem F_{t+1}(-y),\quad g_t''(y)=- \pem f_{t+1}(-y).
    \end{align*}
    Hence $g_t\in ACD(\R).$ Thus, by  \cref{lem:clipped-expectation} applied to $V_{t+1}$ and the induction hypothesis  $h_t \in \ACD(\R)$ and 
	\begin{align} \nonumber
	\holdt'(y) &=  \pem   F_{t+1}(-y) + \Exv\big[ V_{t+1}'(Z_{t+1} + y) \1_{\{0 \le Z_{t+1} + y \le \PufferStoremax\}} \big]  \\ \label{eq:A1} &=
		\begin{aligned}[t]
			&\big(\pem - V_{t+1}'(0)\big) F_{t+1}(-y) +  V_{t+1}'(\PufferStoremax) F_{t+1}(\PufferStoremax-y) \\
			& - \int_{0}^{\mathrlap{\PufferStoremax}} V_{t+1}''(z) F_{t+1}(z-y) \diff z, 
		\end{aligned} \\ \label{eq:A2} 
	\holdt''(y) &= 
	\begin{aligned}[t]
		&-\big(\pem - V_{t+1}'(0)\big) f_{t+1}(-y) -  V_{t+1}'(\PufferStoremax) f_{t+1}(\PufferStoremax-y) \\
		& + \int_{0}^{\mathrlap{\PufferStoremax}} V_{t+1}''(z) f_{t+1}(z-y) \diff z, 
	\end{aligned} 
	\end{align}
	Further, we have $ \pout \le  V_{t+1}'\le \pin$ and  $ V_{t+1}''(z) \le 0 $ a.e.  Hence, all terms on the right-hand side of \eqref{eq:A2} are non-positive. Thus, $h''_t \le 0$ a.e.\ which implies for $x<y$
\begin{align*}
    h'_t(y)-h'_t(x) = \int_x^y h''_t(z)dz \le 0.
\end{align*}
Thus $h_t'$ is non-increasing and hence  $ \holdt $ is concave on $ \R $.	

    From \eqref{eq:A1} and dominated convergence  we get
	\begin{align*}
	\lim\limits_{y \to - \infty} h_t'(y) = \pem, \quad   \lim\limits_{y \to + \infty} h_t'(y)=0.
	\end{align*}

Since \(h_t'\) is continuous and non-increasing and
$0<\pout\le\pin<\pem,$
the equations
\[
h_t'(y)=\pin
\quad\text{and}\quad
h_t'(y)=\pout
\]
have non-empty solution sets. In particular the level sets $\{y : h'_t(y)=\pin\}$ and $\{y : h'_t(y)=\pout\}$ are non-empty, closed and bounded. Thus, minimal and maximal elements exist.
Define
\begin{align*}
	\tthcht \coloneqq \min \{y \in \R \colon h_t'(y) = \pin\}
	\end{align*}
and
  \begin{align*}
	\tthdst \coloneqq \max \{y \in \R \colon h_t'(y) = \pout \}.
	\end{align*}
Since $\holdt'$ is monotonically non-increasing and $\pout \le \pin$, we get $\tthcht \le \tthdst.$ The constrained maximum points are therefore
\begin{align*}
    \thcht \coloneqq \clipstorage{\tthcht} \quad\mbox{ and }  \quad \thdst \coloneqq \clipstorage{\tthdst}
\end{align*}
and again ordered.

	We will now show that the threshold policy with thresholds $\thcht$ and $\thdst$ is an optimal policy. We consider the post-state-value-function
	\begin{equation}\label{eq:Qf}
	Q_t(x,y) = \begin{cases}
	\hdst(y) + \pout x, &y \le x,\\
	\hcht(y) + \pin x, &y > x.
	\end{cases}
	\end{equation}
	An optimal decision rule $\pol^*_t\colon \statespace \to \statespace$ is a measurable map that maximizes $Q_t$ with respect to the post-state $y$ for every state $x \in \statespace$, i.e.
	\begin{align*}
	\pol^*_t(x) \in \argmax_{y \in [0,\PufferStoremax]} Q_t(x,y) \quad (x \in \statespace).
	\end{align*}
	We distinguish three cases.

	\begin{description}
		\item[1.\ Case: $\thcht \le x \le \thdst$] 
        Since \(\overline h_t\) is non-decreasing on
\((-\infty,\tthdst]\), it follows that \(y=x\) maximizes the
first branch of \eqref{eq:Qf} over \(y\in[0,x]\).
Similarly, since \(\underline h_t\) is non-increasing on
\([\tthcht,\infty)\), the point \(y=x\) maximizes the second
branch over \(y\in[x,\PufferStoremax]\).
Hence
\[
\pi_t^*(x)=x.
\]
        
		\item[2.\ Case: $0 \le x < \thcht$] 
        Then
\[
\max_{y\le x}Q_t(x,y)=Q_t(x,x).
\]
On the other hand,
\begin{align*}
Q_t(x,\thcht)
&=
\underline h_t(\thcht)+\pin x
\ge
\underline h_t(x)+\pin x
=
Q_t(x,x).
\end{align*}
Thus an optimal post-state is
\[
\pi_t^*(x)=\thcht.
\]

		\item[3.\ Case: $\thdst < x \le \PufferStoremax$] The argument is analogous and gives $\pol^*_t(x) = \thdst$.
	\end{description} 
Combining the three cases proves
\[
\pi_t^*(x)
=[x]^\thds{t}_\thch{t}.
\]
 To complete the induction it remains to show that $V_t \in \ACD[0,\PufferStoremax]$, $V_t$ is concave and $V_t' \in [\pout,\pin]$. We have
	\begin{equation} \label{eq: optimal value function structure}
	V_t(x) = Q_t\big(x,\pol^*_t(x)\big) = \begin{cases}
	\pin x + \hcht(\thcht), & 0 \le x < \thcht, \\
	\holdt(x), & \thcht \le x \le \thdst, \\
	\pout x + \hdst(\thdst), & \thdst < x \le \PufferStoremax.
	\end{cases}
	\end{equation}
	Obviously, $V_t$ is continuous on $[0,\PufferStoremax]$ and we obtain
    \begin{equation} \label{eq: ac_derivative of V_t (abs cont)}
	V_t'(x) = \begin{cases}
	\pin, & 0 < x < \thcht, \\
	\holdt'(x), & \thcht < x < \thdst, \\
	\pout, & \thdst < x < \PufferStoremax.
	\end{cases}
	\end{equation}
    Whenever $\thcht\in(0,\PufferStoremax),$ clipping is inactive and therefore $h_t'(\thcht)=\pin.$ Likewise, whenever $\thdst \in(0,\PufferStoremax)$ we obtain $h'_t(\thdst)=\pout.$ Thus, the adjacent pieces meet continuously at every interior threshold point. At a threshold point equal to $0$ or $\PufferStoremax,$ only the appropriate one-sided derivative is relevant. Consequently, $V_t\in C^1([0,\PufferStoremax])$ in the standard one-sided sense at the endpoints.   Since $h_t'$ is absolutely continuous, the resulting derivative $V'_t$ is absolutely continuous on $[0,\PufferStoremax]$ and $V_t \in \ACD[0,\PufferStoremax]$. Moreover, because $h'_t$ is non-increasing and $\pout \le h_t'\le \pin$, $V_t'$ is non-increasing on the whole interval with derivative in $[\pout,\pin]$. Therefore,  $V_t$ is concave.

     This concludes the induction step.
\end{proof}

\section{Proof of  \cref{rem: properties of the value function under Lipschitz assumption}}\label{sec:appendixB}
\begin{proof}
The proof is quite similar to the one of  \cref{thm: optimal policies are threshold policies (abs cont)}. We only prove \eqref{eq: Lipschitz constant of value function (abs cont)} here.
	Let $I_{t+1}\colon \R \to \R, y \mapsto \int_{0}^{\PufferStoremax} V_{t+1}''(z) F_{t+1}(z-y) \diff z$. For all $y_1,y_2 \in \R$ we have
	\begin{align*}
		\abs{I_{t+1}(y_2) - I_{t+1}(y_1)} &\le \int_{0}^{\PufferStoremax} \abs{V_{t+1}''(z)} \abs{F_{t+1}(z-y_2) - F_{t+1}(z-y_1)} \diff z \\
		&\le  \int_{0}^{\PufferStoremax} \abs{V_{t+1}''(z)} \diff  z \, L_{t+1} \abs{y_2 - y_1} \\
		&=  \int_{0}^{\PufferStoremax} -V_{t+1}''(z) \diff z \, L_{t+1} \abs{y_2 - y_1} \\
		&=  \big(V_{t+1}'(0)-V_{t+1}'(\PufferStoremax)\big) L_{t+1} \abs{y_2 - y_1}. 
	\end{align*}
	Hence,  we get 
	\begin{align*}
		\Lip(h'_t) &\le \abs{\pem - V'_{t+1}(0)} L_{t+1} + \abs{V'_{t+1}(\PufferStoremax)} L_{t+1} + \big(V'_{t+1}(0)-V'_{t+1}(\PufferStoremax)\big) L_{t+1} \\
		&= \Big( \pem - V'_{t+1}(0) + V'_{t+1}(\PufferStoremax) + V'_{t+1}(0)-V'_{t+1}(\PufferStoremax)\Big) L_{t+1}\\
		&= \pem L_{t+1}.
	\end{align*}

    It remains to transfer the Lipschitz property of $h_t'$ to $V_t'$.
By \cref{eq: ac_derivative of V_t (abs cont)},
\[
V_t'(x)
=
\begin{cases}
p_{\rm in},
&0\le x<\underline\vartheta_t^*,\\
h_t'(x),
&\underline\vartheta_t^*
\le x\le\overline\vartheta_t^*,\\
p_{\rm out},
&\overline\vartheta_t^*<x\le P_{\max}.
\end{cases}
\]
If a threshold lies in the interior, the adjacent pieces agree at
that threshold. Thus $V_t'$ is continuous.

Let $0\le x<y\le P_{\max}$. Split the interval $[x,y]$, if
necessary, at the two thresholds. The function $V_t'$ is constant
on the parts outside
$[\underline\vartheta_t^*,\overline\vartheta_t^*]$, whereas on the
middle part it coincides with $h_t'$. Hence
\[
|V_t'(y)-V_t'(x)|
\le
\Lip(h_t')|y-x|.
\]
The same argument applies when one or both thresholds are located at
the boundary. Therefore,
\[
\Lip(V_t')
\le
\Lip(h_t')
\le
p_{\rm em}L_{t+1}.
\]
This proves the assertion.
\end{proof}

\section{Proof of the Results in \cref{sec:algorithm}}\label{sec:appendixC}

Beforewe state the proofs for \cref{sec:algorithm}, we introduce some additional notation. Similar to $\hcht$ and $\hdst$, we write
\begin{align*}
    \hchpara{t}(y) &\coloneqq \hpara{t}(y) - \pin y , &\achpara{t} &\coloneqq (\hchpara{t})'(y) = \apara{t} - \pin,\\
    \hdspara{t}(y) &\coloneqq \hpara{t}(y) - \pout y,
    & \adspara{t} &\coloneqq (\hdspara{t})'(y) = \apara{t} - \pout
\end{align*}
for $y \in \R$ and $t \in \TimeIndex'$.



\begin{proof}[Proof of \cref{lem:derivative in expectation}]
    Follows from \cref{lem:pathhderivative,rmk: switching expectation value and derivatie}.
\end{proof}

\begin{proof}[Proof of \cref{lem:pathhderivative}]
As a preliminary, it holds 
\[ r^{(1)}_{t}(x, \clip{x}{\lt_t}{\ut_t} )= -\pin (\lt_t-x)^+ + \pout (\ut_t-x)^-, \quad x \in \statespace. \]

a) Let $\ths \in \thsspace$. We define $\forbiddenset{T} \coloneqq \{0,\PufferStoremax\}$ and $\forbiddenset{t} \coloneqq \criticalset{t}$ for $t \in \TimeIndex'$. Further, let  
    \begin{equation*} 
        \goodset{t}(y) \coloneqq \begin{aligned}[t]
            \{\episode \in \episodespace \ \vert \  &\lt_a + \textstyle{\sum_{k=a}^{b}} z_{k+1} \not \in \forbiddenset{b+1} \text{ for }  t+1 \le a < b \le T-1,\\
            &y + \textstyle{\sum_{k=t}^{b'}} z_{k+1} \not \in \forbiddenset{b'+1} \text{ for }  t \le b' \le T-1,\\
            &\ut_a + \textstyle{\sum_{k=a}^{b}} z_{k+1} \not \in \forbiddenset{b+1} \text{ for }  t+1 \le a < b \le T-1 \}
        \end{aligned} 
    \end{equation*}
    for $y \in \R$ and $t \in \TimeIndex'$. We have $\PP\big(\Episode \in \goodset{t}(y)\big)=1 $ for all $t \in \TimeIndex'$ and $y \in \R$ due to the distribution assumption on $(Z_{t+1})$ in \cref{ass:PG}. Therefore, it is sufficient to show by backwards induction over $t \in \TimeIndex'$ that for all $y \in \R$ and  $\episode \in \goodset{t}(y)$ the functions $\hsmp{t}$ is differentiable at $y$ with the derivative given in \cref{eq:pathwise-a}. For $t=T-1$, let $y \in \R$ and $\episode \in \goodset{T-1}(y)$. It holds
\[
\hsmp{T-1}(y)=-\pem(y+z_T)^-+\pout[y+z_T]_0^{P_{\max}}.
\]
With $\goodset{T-1}(y) = \big\{\episode \in \episodespace \mid y+z_T \not \in \{0,\PufferStoremax\}\big\}$ the kink points are away from $y+z_T$ and
\[
(\hsmp{T-1})'(y)=\pem\1_{\{y+z_T<0\}}+\pout\1_{\{0<y+z_T<P_{\max}\}}.
\]

Now suppose that for all $\tilde y \in \R$ and all $\episode \in \goodset{t+1}(\tilde y)$
the derivative formula is valid at time $t+1$. Let $y \in \R$ and $\episode \in \goodset{t}(y)$. Plugging $\Vsmp{t+1}$ into the definition of $\hsmp{t}$ yields the recursion
\begin{align*}
    \hsmp{t}(y)&=-p_{\rm em}(y+z_{t+1})^-+\Vsmp{t+1}([y+z_{t+1}]_0^{P_{\max}})\\
    &= -p_{\rm em}(y+z_{t+1})^-+r^{(1)}_{t+1}( y + z_{t+1}, \clip{y + z_{t+1}}{\lt_{t+1}}{\ut_{t+1}} ) + \hsmp{t+1}(\clip{y+z_{t+1}}{\lt_{t+1}}{\ut_{t+1}}).
\end{align*}
In particular, $y + z_{t+1}$ is away from $ \criticalset{t+1}$ and therefore, the first two terms are differentiable at $y$. Moreover, $ \episode $ is contained in the sets $\goodset{t+1}(\lt_{t+1})$, $\goodset{t+1}(y + z_{t+1})$ and $ \goodset{t+1}(\ut_{t+1})$, so 
$$\episode \in \goodset{t+1}(\clip{y+z_{t+1}}{\lt_{t+1}}{\ut_{t+1}}).$$
Therefore, $\hsmp{t+1}$ is differentiable at $\clip{y+z_{t+1}}{\lt_{t+1}}{\ut_{t+1}}$ and the chain rule gives
\begin{align} \label{eq: realized aprime}
\asmp{t}(y) &=
    \begin{aligned}[t]
        &p_{\rm em}\1_{\{y+z_{t+1}<0\}} + p_{\rm in} \1_{\{0<y+z_{t+1}<\underline\vartheta_{t+1}\}} \\
    &+\asmp{t+1}(y+z_{t+1})\1_{\{\underline\vartheta_{t+1}<y+z_{t+1}<\overline\vartheta_{t+1}\}} \\
    &+p_{\rm out}\1_{\{\overline\vartheta_{t+1}<y+z_{t+1}<P_{\max}\}}.
    \end{aligned}
\end{align}

b) Similar to a) let $\forbiddensettilde{t} \coloneqq \{\ltt, \utt\} \setminus \{0,\PufferStoremax\}$ for $t \in \TimeIndex'$ and
\begin{equation*} 
        \goodsettilde(\ths) \coloneqq \bigcap_{t\in\TimeIndex'} \begin{aligned}[t]
            \{ \episode \in \episodespace \ \vert \ & \lt_{a} + \textstyle{\sum_{k = a}^{t-1}} \dist{k+1} \not \in \forbiddensettilde{t} \text{ for } 0 \le a \le t-1,\\
            &x_0 + \textstyle{\sum_{k=0}^{t-1}} \dist{k+1} \not \in \forbiddensettilde{t},\\
            &\ut_{a} + \textstyle{\sum_{k = a}^{t-1}} \dist{k+1} \not \in \forbiddensettilde{t} \text{ for } 0 \le a \le t-1 \},
        \end{aligned}
    \end{equation*}
    for $\ths \in \thsspace$. Again $\PP\big(\Episode \in \goodsettilde(\ths)\big) = 1$ for all $\ths \in \thsspace$. For the partial derivative with respect to the lower threshold $\lt_t$, write the
realized total reward as
\[
\lsmp(\ths)
= C_t^e + r_t^{(1)}\left(\xsmp{t},[\xsmp{t}]_{\lt_t}^{\ut_t}\right) + \hsmp{t} \left([\xsmp{t}]_{\lt_t}^{\ut_t}\right),
\]
where $C_t^e$, $\xsmp{t}$ only depend on the past parameters $ \ths_0 , \dots, \ths_{t-1}$ and not on $\ths_t$. Similarly, $\hsmp{t}$ only depends on the future parameters $\ths_{t+1}, \dots, \ths_{T-1}$ and not on $\ths_t$. The function $\thsspace_t \to \R, \ths_t \mapsto \clip{x}{\lt_t}{\ut_t}$ for $x \in \statespace$ is differentiable (relative to $\thsspace_t$) with 
\[ \partial_{\lt_t} \clip{x}{\lt_t}{\ut_t} = \1_{\{x \le \lt_t, x \not = \PufferStoremax\}}. \]
For $e \in \goodsettilde(\ths)$ the sample state $\xsmp{t}$ can only be $x_0 + \sum_{k=0}^{t-1} z_{k+1}$, if the sample state path always stayed between the thresholds, or it is equal to $\lt_{a} + \sum_{k = a}^{t-1} \dist{k+1}$ or $\ut_{a} + \sum_{k = a}^{t-1}\dist{k+1}$ for $ 0 \le a \le t-1$, if a threshold was breached. Therefor, $e \in \goodset{t}(\xsmp{t})$ and with the further observation $e \in \goodset{t}(\lt_t) \cap \goodset{t}(\ut_t)$, the function $\hsmp{t}$ is differentiable at $[\xsmp{t}]_{\lt_t}^{\ut_t}$. By the chain rule 
$$\partial_{\lt_t} \lsmp (\ths) = G_{t,-}(\ths,\episode).$$
The upper-threshold formula follows analogously.

\end{proof}



For later use define
\[
F_t^{\vartheta,-}(a)
\coloneqq
\PP_\rho^\vartheta
(X_t\le a,\ X_t\neq\Pmax)
\]
and
\[
F_t^{\vartheta,+}(a)
\coloneqq
\PP_\rho^\vartheta
(X_t\ge a,\ X_t\neq0).
\]
\begin{remark}
    \cref{ass:PG} implies
    \begin{equation}
    F_t^{\vartheta,-}(a)\ge\alpha,
    \qquad
    F_t^{\vartheta,+}(a)\ge\alpha,
    \qquad
    a\in[0,\Pmax].
    \label{eq:activation-lower-bound}
    \end{equation}
    These inequalities can be seen as follows: For $t=0$,
    \[
    F_0^{\vartheta,-}(a)
    \ge\PP(X_0=0)\ge\alpha_0,
    \qquad
    F_0^{\vartheta,+}(a)
    \ge\PP(X_0=P_{\max})\ge\alpha_0.
    \]
    For $t\ge1$, independently of the previous post-state, $Z_t\le-P_{\max}$ implies $X_t=0$ and $Z_t\ge P_{\max}$ implies $X_t=P_{\max}$.
    Thus
    \begin{align*}
        F_t^{\vartheta,-}(a)\ge\alpha_Z,
    \qquad
    F_t^{\vartheta,+}(a)\ge\alpha_Z,
    \end{align*}
    for all $a\in[0,P_{\max}]$. Hence \eqref{eq:activation-lower-bound} follows with
    $\alpha=\min\{\alpha_0,\alpha_Z\}$.
\end{remark}

\begin{proof}[Proof of \cref{rmk: switching expectation value and derivatie}]
    The functions $\hsmp{t}$, $t \in \TimeIndex$, and $\lsmp$ are Lipschitz-continuous due to $r^{(1)}$, $r^{(2)}$ and $(x,\lt_t,\ut_t) \mapsto \clip{x}{\lt_t}{\ut_t}$ being Lipschitz-continuous in all arguments. Applying the dominated convergence theorem yields 
    \begin{equation*}
    \partial_{\lt_t}\lo(\ths)
    = \EE_\Episode[\partial_{\lt_t} \lrv(\ths)]  = F_t^{\vartheta,-}(\uth_t)
    \big(a_t^\vartheta(\uth_t)-\pin\big),
    \end{equation*}
    \begin{equation*}
    \partial_{\ut_t} \ell(\vartheta) = \EE_\Episode[\partial_{\ut_t} \lrv(\ths)] =  F_t^{\vartheta,+}(\oth_t) \big(a_t^\vartheta(\oth_t)-\pout\big),
    \end{equation*}
    and
    $$\apara{t}(y) = \EE_\Episode[\arv{t}(y)],\qquad  t \in \TimeIndex'.$$
\end{proof}

\begin{remark} With \cref{lem:derivative in expectation} it follows that $\Vpara{t}$ is differentiable on $\statespace \setminus \big(\{\lt_{t},\ut_t\} \cap (0,\PufferStoremax)\big)$ with
\begin{align} \label{eq: Vprime}
    ( \Vpara{t})'(x)=
    \begin{cases}
    \pin, & 0 \le x<\lt_t,\\
    \apara{t}(x), & \lt_t < x < \ut_t,\\
    \pout, & \ut_t<x\le \PufferStoremax.
    \end{cases}  
\end{align}
In particular, if a threshold is on the boundary $\{0,\PufferStoremax\}$ the function $\Vpara{t}$ is differentiable (relative to $\statespace$) on that boundary.
\end{remark}

\begin{remark}
    The following bounds hold
    \begin{equation} \label{eq:pathwise-bound}
        0 \le \asmp{t}(y), \apara{t}(y) \le \pem,
    \end{equation}
    for all $y \in \R, \ths \in \thsspace, \episode \in \episodespace$ and $ t \in \TimeIndex'$. The inequalities can be seen by backwards induction over \cref{eq: aprime,eq:pathwise-a}.
\end{remark}



\begin{proof}[Proof of  \cref{lem:unbiased}]
For one trajectory set \(G(\vartheta,E)\) as in \cref{lem:pathhderivative}, and for a
batch of \(N\) independent trajectories define
\[
 \widehat\nabla_N\ell(\vartheta)
 \coloneqq\frac1N\sum_{n=1}^N G(\vartheta,E^{(n)}).
\]
\cref{eq:grad-lower,eq:grad-upper} show that
\[
 \E[G(\vartheta,E)]=\nabla\ell(\vartheta),
 \qquad
 \E[\widehat\nabla_N\ell(\vartheta)]=\nabla\ell(\vartheta).
\]
Moreover, \eqref{eq:pathwise-bound} implies that every coordinate of
\(G(\vartheta,E)\) has absolute value at most \(\psum\).  Hence
\[
 \E\norm{G(\vartheta,E)-\nabla\ell(\vartheta)}_2^2
 =\sum_{j=1}^{2T}\operatorname{Var}(G_j(\vartheta,E))
 \le 2T\psum^2.
\]
Averaging \(N\) independent copies gives
\begin{equation}
 \E\norm{\widehat\nabla_N\ell(\vartheta)-\nabla\ell(\vartheta)}_2^2
 \le \frac{\sigma^2}{N},
 \qquad
 \sigma^2\coloneqq2T\psum^2.
 \label{eq:43}
\end{equation}

For the stochastic algorithm we need the corresponding conditional statement.  We obtain 
\begin{align*}
 \E\!\left[
   \widehat\nabla_N\ell(\vartheta^{(k)})\mid\mathcal F_k
 \right]
 &=\nabla\ell(\vartheta^{(k)}),\\
 \E\!\left[
  \norm{\widehat\nabla_N\ell(\vartheta^{(k)})
        -\nabla\ell(\vartheta^{(k)})}_2^2
   \mid\mathcal F_k
 \right]
 &\le \frac{\sigma^2}{N}.
\end{align*}
This is the form used in the proof of \cref{thm:PG}.


\end{proof}

Proof of  \cref{lem:smooth}.
\begin{proof}
    We first establish four elementary estimates.

\medskip
\noindent\emph{Step 1: Lipschitz continuity of the state distribution functions.}
For \(0\le a<b\le\PufferStoremax\),
\[
 F_t^{\vartheta,-}(b)-F_t^{\vartheta,-}(a)
 =\PP_\rho^\vartheta(a<X_t\le b,\ X_t\ne\PufferStoremax).
\]
For \(t=0\), the restriction of the law of \(X_0\) to
\((0,\PufferStoremax)\) has density bounded by \(L_\rho\).  For \(t\ge1\), let
\(Y_{t-1}=[X_{t-1}]_{\uth_{t-1}}^{\oth_{t-1}}\).  Conditional on
\(Y_{t-1}\), the restriction of the law of \(X_t\) to \((0,\PufferStoremax)\) has density
\[
 x\longmapsto f_t(x-Y_{t-1}),
\]
and therefore the unconditional interior density is bounded by \(L_Z\).  The same
argument applies to \(F_t^{\vartheta,+}\).  Thus, uniformly in \(\vartheta\),
\[
 \Lip(F_t^{\vartheta,-})\le\Lambda,
 \qquad
 \Lip(F_t^{\vartheta,+})\le\Lambda.
\]
Notice that the atoms at \(0\) and \(\PufferStoremax\) do not enter these increments because of
the exclusions in the definitions of \(F_t^{\vartheta,\pm}\).

\medskip
\noindent\emph{Step 2: Lipschitz continuity of \(a_t^\vartheta\) in its argument.}
Set
\[
 M_t\coloneqq\sup_{\vartheta\in\Theta}\Lip(a_t^\vartheta).
\]
For \(t\le T-2\), \(a^\ths_t\) is recursively given by \eqref{eq: aprime} with
\begin{align*}
 a_t^\vartheta(y)
 &=\pem F_{t+1}(-y)
   +\pin\PP(0<y+Z_{t+1}<\uth_{t+1})\\
 &\quad
   +\E\!\left[
      a_{t+1}^\vartheta(y+Z_{t+1})
      \1_{\{\uth_{t+1}<y+Z_{t+1}<\oth_{t+1}\}}
    \right]\\
 &\quad
   +\pout\PP(\oth_{t+1}<y+Z_{t+1}<\PufferStoremax).
\end{align*}
For $y_1,y_2\in[0,\PufferStoremax]$ let \(d=|y_2-y_1|\). We consider $|a_t^\vartheta(y_1)-a_t^\vartheta(y_2)|$.  The first, second and fourth terms change by at most
\(\pem L_Zd\), \(2\pin L_Zd\) and \(2\pout L_Zd\), respectively.
For the middle term, on the event on which both shifted arguments lie in the inaction
interval we use \(M_{t+1}d\); on the symmetric difference of the two inaction events we
use \eqref{eq:pathwise-bound} and the fact that this symmetric difference is contained in two
intervals of total probability at most \(2L_Zd\).  Hence 
\[
 M_t\le M_{t+1}+(3\pem+2\pin+2\pout)L_Z
 \le M_{t+1}+3\psum L_Z.
\]
The same estimate holds directly for \(t=T-1\).  Since \(M_T=0\), backward induction
yields

\begin{equation}
 M_t\le 3\psum L_Z(T-t).
 \label{eq:C9}
\end{equation}

\medskip
\noindent\emph{Step 3: sensitivity of the continuation derivative to one policy block.}
Suppose \(\vartheta,\vartheta'\in\Theta\) differ only in the \(k\)-th block.  If
\(t<k\), then
\begin{equation}
 \sup_{y\in\R}
 |a_t^\vartheta(y)-a_t^{\vartheta'}(y)|
 \le \psum L_Z\norm{\vartheta_k-\vartheta_k'}_1.
 \label{eq:C10}
\end{equation}
To prove this, observe first that the continuation after time \(k\) is identical, hence
\(a_k^\vartheta=a_k^{\vartheta'}\), \(\underline a^\ths_k = \underline a^{\ths'}_k\) and \(\overline a^\ths_k = \overline a^{\ths'}_k\).  On \((0,\PufferStoremax)\), with \eqref{eq: Vprime} we obtain
\begin{align*}
    (V_k^\vartheta)'(x) 
 &=a_k^\ths(x) + (\pin-a_k^\ths(x))\1_{\{x\le\uth_k\}}
       +(\pout-a_k^\vartheta(x))\1_{\{x\ge\oth_k\}}\\
       &=a_k^\ths(x) - \underline a^\ths_k(x)\1_{\{x\le\uth_k\}}
       -\overline a_k^\vartheta(x)\1_{\{x\ge\oth_k\}},
\end{align*}
with an analogous expression for \(\vartheta'\).  Therefore the two derivatives differ
only on the intervals between the two lower thresholds and between the two upper
thresholds, and
\begin{align}
    |(V_k^\vartheta)'(x)-(V_k^{\vartheta'})'(x)|\le
    \begin{aligned}[t]
        &|\underline a^\ths_k(x)|\1_{\big\{\min\{\lt_k,\lt_k'\} \le x \le \max\{\lt_k,\lt_k'\}\big\}}  \\
    &+ |\overline a^\ths_k(x)|\1_{\big\{\min\{\ut_k,\ut_k'\} \le x \le \max\{\ut_k,\ut_k'\}\big\}}.
    \end{aligned}
    \label{eq: sensitivity of value function derivitive to policy}
\end{align}
Using the density bound for \(Z_k\),
\begin{equation}
    \sup_y|a_{k-1}^\vartheta(y)-a_{k-1}^{\vartheta'}(y)|
 \le L_Z\int_0^{\PufferStoremax}
 |(V_k^\vartheta)'(x)-(V_k^{\vartheta'})'(x)|\,dx. \label{eq: sup-norm of a bound by derivitive of value func}
\end{equation}
For earlier times $t < k-1$ the thresholds at time
\(t+1\) are the same under the two parameters $\ths$ and $\ths'$. Thus the difference of the derivatives
of the corresponding value functions is zero outside the common inaction interval $[\lt_t, \ut_t].$ 
Therefore, for $t < k-1$, we have
\begin{align*}
    |a_t^\ths(y) - a_t^{\ths'}(y)| \le \EE|a_{t+1}^\ths(y+Z_{t+1}) - a_{t+1}^{\ths'}(y+Z_{t+1})| \le \sup_{\tilde y} |a_{t+1}^\ths(\tilde y) - a_{t+1}^{\ths'}(\tilde y)|.
\end{align*}
By backward induction we arrive at
\begin{align} \label{eq: inequality sup-norm difference in a}
    \sup_y|a_t^\ths(y) - a_t^{\ths'}(y)| \le L_Z \left(  \int_{\min\{\uth_k,\uth_k'\}}^{\max\{\uth_k,\uth_k'\}}
   |\underline a^\ths_k(x)|\,dx+
 \int_{\min\{\oth_k,\oth_k'\}}^{\max\{\oth_k,\oth_k'\}}
   |\overline a^{\ths'}_k(x)|\,dx \right).
\end{align}
The functions $\underline a^\ths_k$ and $\overline a^\ths_k$ are uniformly bounded by $\psum$, so \eqref{eq:C10} is proven for all $t < k$.

\medskip
\noindent\emph{Step 4: sensitivity of the state distribution to one policy block.}
If \(t>k\), then \(a_t^\vartheta=a_t^{\vartheta'}\), because the continuation
thresholds from \(t+1\) onward coincide.  Couple the two state processes by the same
initial state and the same disturbances.  The clipping maps satisfy
\[
 \left|
 [x]_{\uth_k}^{\oth_k}-[x]_{\uth_k'}^{\oth_k'}
 \right|
 \le \norm{\vartheta_k-\vartheta_k'}_1,
\]
and all subsequent clipping maps are nonexpansive.  Hence, for every \(s\ge k\),
\[
 |Y_s^\vartheta-Y_s^{\vartheta'}|
 \le \norm{\vartheta_k-\vartheta_k'}_1
 \quad\text{a.s.}
\]
For \(a\in[0,\PufferStoremax]\), conditioning on the post-state at time \(t-1\) and using the $L_Z$ -Lipschitz continuity of \(F_t\) gives, also at the endpoints because \(Z_t\) has a density,
\[
 |F_t^{\vartheta,\pm}(a)-F_t^{\vartheta',\pm}(a)|
 \le L_Z\E|Y_{t-1}^\vartheta-Y_{t-1}^{\vartheta'}|.
\]
Consequently,
\begin{equation}
 \sup_{a\in[0,\PufferStoremax]}
 |F_t^{\vartheta,\pm}(a)-F_t^{\vartheta',\pm}(a)|
 \le L_Z\norm{\vartheta_k-\vartheta_k'}_1.
 \label{eq:C11}
\end{equation}

We now compare the gradients using \eqref{eq:grad-lower}--\eqref{eq:grad-upper}.  If \(t<k\), the
state distribution at time \(t\) is the same under \(\vartheta\) and \(\vartheta'\), so
\eqref{eq:C10} and \(\norm{v}_1\le\sqrt2\norm{v}_2\) imply
\[
 \norm{\nabla_{\vartheta_t}\ell(\vartheta)
       -\nabla_{\vartheta_t}\ell(\vartheta')}_2
 \le2\psum L_Z\norm{\vartheta_k-\vartheta_k'}_2.
\]
If \(t>k\), the continuation derivative is the same and the same estimate follows from
\eqref{eq:C11}.  Finally, for \(t=k\), the functions
\(F_k^{\vartheta,\pm}=F_k^{\vartheta',\pm}\) and
\(a_k^\vartheta=a_k^{\vartheta'}\), but their arguments change.  From the Lipschitz estimate for $F_k^{\vartheta,\pm}$ above and \eqref{eq:C9},
\[
 \norm{\nabla_{\vartheta_k}\ell(\vartheta)
       -\nabla_{\vartheta_k}\ell(\vartheta')}_2
 \le \psum(\Lambda+3L_ZT)
       \norm{\vartheta_k-\vartheta_k'}_2
 \le4\psum\Lambda T\norm{\vartheta_k-\vartheta_k'}_2.
\]
Thus, whenever two policy parameters differ in one time block only,
\[
 \norm{\nabla\ell(\vartheta)-\nabla\ell(\vartheta')}_2
 \le4\psum\Lambda T^{3/2}
     \norm{\vartheta_k-\vartheta_k'}_2.
\]
For arbitrary \(\vartheta,\vartheta'\), change the \(T\) blocks one after another and
apply the triangle inequality followed by Cauchy--Schwarz.  We obtain
\begin{equation}
 \norm{\nabla\ell(\vartheta)-\nabla\ell(\vartheta')}_2
 \le 4\psum\Lambda T^2\norm{\vartheta-\vartheta'}_2.
 \label{eq:44}
\end{equation}
Hence \(\ell\) has an \(L\)-Lipschitz gradient on \(\Theta\) with
\(L=4\psum\Lambda T^2\).

\end{proof}

For the proof of \cref{lem:KL} we use \cite[Theorem 1]{chen2026LandscapePolicyOptimization}. Our notation differs from that in \cite{chen2026LandscapePolicyOptimization} therefore, we restate Theorem 1 after introducing the notation:

Let $\ths \in \thsspace$ be a policy parameter and $t \in \TimeIndex'$. We split $\ths$ into the prefix and suffix policy parameters $\thspa{t}$  and $\thsfu{t}$ given by 
\[ \thspa{t} = \big( \lt_0, \ut_0,\dots, \lt_{t-1}, \ut_{t-1} \big) \quad \text{and} \quad \thsfu{t} = \big(\lt_{t+1}, \ut_{t+1}, \dots,\lt_{T-1}, \ut_{T-1} \big). \]
So, we can write $\ths=(\thspa{t}, \ths_t, \thsfu{t})$. For a fixed policy parameter $\ths$ we define two functions
\begin{align}
    \Psi_t^\ths(\eta) &\coloneqq \E_\rho^{\thspa{t}}\left[Q^{\thsfu{t}}_t \left( X_t,[X_t]_{\underline\eta}^{\overline\eta} \right)\right],  &\eta&=(\underline\eta,\overline\eta)\in\Theta_t, \label{eq:Psi} \\
    \Phi_t^\ths(\eta) &\coloneqq \E_\rho^{\thspa{t}}\left[Q_t \left( X_t,[X_t]_{\underline\eta}^{\overline\eta} \right)\right], \qquad &\eta&=(\underline\eta,\overline\eta)\in\Theta_t, \label{eq:Phi}
\end{align}
where $Q_t^{\thsfu{t}}$ is the parameterized $Q$-function and $Q_t$ is the optimal $Q$-function. We remark that $\Psi^\ths_t$ does not depend on $\ths_t$ and $\Phi^\ths_t$ only depends on $\thspa{t}$, not on $\ths_t$ and $\thsfu{t}$. Further, if $\thsfu{t}$ is optimal, i.e.\ $\thsfu{t} = \thsfu{t}^*$, then $\Psi_t^\ths = \Phi_t^\ths$ on $\thsspace_t$.
\begin{theorem}[Theorem 1 in \cite{chen2026LandscapePolicyOptimization}] \label{thm: Theorem cited from Landscape Paper}
    Suppose the following conditions hold.
    \begin{enumerate}[1)]
        \item \textbf{(Bounded Gradients)} For all $t \in \TimeIndex'$ and $\ths \in \thsspace$, the function $\Psi^\ths_t$ is continuously differentiable on $\thsspace_t$ with the $2$-norm of its gradient upper bounded by $G$.
        \item \textbf{(\PLK condition of Expected Optimal $Q$-value Functions)} For all $t \in \TimeIndex'$ and $\ths \in \thsspace$, the function $\Phi^\ths_t$ satisfies the \PLK condition with a \PLK constant $\mu_Q > 0$ on $\thsspace_t$.
        \item \textbf{(Sequential Decomposition Inequality)} There exists a $M_g>0$ such that for all $0 \le t < k \le T-1$ and $\ths=(\thspa{t}, \ths_t, \thsfu{t}) \in \thsspace$ the inequality 
        \begin{align*}
            \norm{\nabla_{\ths_t}\lo\big(\thspa{k}, \ths_k, \thsfu{k}^* \big) - \nabla_{\ths_t}\lo\big(\thspa{k}, \ths_k^*, \thsfu{k}^* \big) }_2 \le M_g\big( \Phi^\ths_k(\ths_k^*) - \Phi^\ths_k(\ths_k)  \big)
        \end{align*}
        holds, where $\ths^* = (\thspa{k}^*, \ths_k^*, \thsfu{k}^*)$ is the optimal policy parameter. 
    \end{enumerate}
    Then the objective function $\lo$ satisfies the \PLK condition on $\thsspace$ with a \PLK constant of 
    \[ \mu_\lo = \frac{\mu_Q^3}{e M_g^2 G^2 T^2}, \]
    i.e.
    \[ \lo^* - \lo(\ths) \le \frac{eM_g^2G^2T^2}{2 \mu_Q^3} \min_{n \in N_\thsspace(\ths)} \norm{\nabla \lo(\ths) - n}^2, \qquad \ths \in \thsspace. \]
\end{theorem}
The following proof of \cref{lem:KL} is similar to the analysis of the stochastic cash balance problem in Section 7 of \cite{chen2026LandscapePolicyOptimization}.

\begin{proof}[Proof of  \cref{lem:KL}]

First, by \eqref{eq:grad-lower},
\eqref{eq:grad-upper}, and $0\le a_t^\vartheta(y)\le\pem,$
we have the uniform estimate
\begin{equation}
\norm{
\nabla_{\vartheta_t}\ell(\vartheta)
}_2
\le
2\psum
=:G.
\label{eq:G}
\end{equation}

Next, we consider the \PLK condition for $\Phi^\ths_t$. For every state $x$ and arbitrary threshold \(\eta=(\underline\eta,\overline\eta)\), with $0 \le \underline \eta \le \overline \eta \le \PufferStoremax$, we can rewrite the $Q_t$-function by
\begin{align*}
    Q_t(x,\clip{x}{\underline \eta}{\overline \eta}) &= Q_t(x, \max\{x, \underline \eta\}) + Q_t(x, \min\{x, \overline \eta\}) - Q_t(x,x)\\
    &= \hcht(\max\{x, \underline \eta\}) + \pin x + \hdst(\min\{x, \overline \eta\}) + \pout x - \holdt(x).
\end{align*}
Therefore, the loss from using the threshold pair
\(\eta=(\underline\eta,\overline\eta)\) instead of the optimal pair can be rewritten as
\begin{align*}
     Q_t\big(x, \clip{x}{\lt_t^*}{\ut_t^*}\big) - Q_t\big(x, \clip{x}{\underline \eta}{\overline \eta}\big) =\begin{aligned}[t]
         &\hcht(\max\{x, \lt^*_t\})  -  \hcht(\max\{x, \underline \eta\})\\
         &+ \hdst(\min\{x, \ut^*_t\}) - \hdst(\min\{x, \overline \eta\}).
     \end{aligned}
\end{align*}
By using the concavity of $\hcht$ and $\hdst$ we get the bounds 
\[ 0 \le  Q_t\big(x, \clip{x}{\lt_t^*}{\ut_t^*}\big) - Q_t\big(x, \clip{x}{\underline \eta}{\overline \eta}\big) \le \hcht(\lt_t^*) - \hcht(\underline \eta) + \hdst(\ut_t^*) - \hdst(\overline \eta).    \]
After taking expectations,
\begin{align}
0&\le \Phi_t^\vartheta(\vartheta_t^*)-\Phi_t^\vartheta(\eta)
\le\hcht(\uth_t^*)- \hcht(\underline\eta)+\hdst(\oth_t^*)-\hdst(\overline\eta).
\label{eq:Q-gap-upper}
\end{align}
By \eqref{eq:A2},  \cref{ass:PG} and the concavity of $V_{t+1}$ for a.e. $y\in [0,\PufferStoremax]$
\begin{align*}
h_t''(y) &=
	\begin{aligned}[t]
		&-\big(\pem - V_{t+1}'(0)\big) f_{t+1}(-y) -  V_{t+1}'(\PufferStoremax) f_{t+1}(\PufferStoremax-y) \\
		& + \int_{0}^{\mathrlap{\PufferStoremax}} V_{t+1}''(z) f_{t+1}(z-y) \diff z, 
	\end{aligned} \\
&\le - \left(\pem- V_{t+1}'(0)+V_{t+1}'(\PufferStoremax) - \int_{0}^{\mathrlap{\PufferStoremax}} V_{t+1}''(z) \diff z \right)\mu_Z \\
&= -\pem \mu_Z=-\kappa.
\end{align*}
Hence both $\hcht$ and $\hdst$ are \(\kappa\)-strongly concave. Their respective maximizers on $\statespace$ are the  optimal lower and upper thresholds $\uth_t^*$ and $\oth_t^*.$
Hence, 
\[
 \hcht(\lt^*_t)-\hcht(\underline \eta)
 \le \hcht'(\underline \eta)(\lt_t^*-\underline \eta)-\frac\kappa2|\lt_t^*- \underline \eta |^2
\]
and an analogous inequality for $\hdst.$ The same independence argument used for \eqref{eq:grad-lower}--\eqref{eq:grad-upper} gives
\begin{equation*}
\partial_{\underline \eta} \Phi_t^\vartheta(\eta) = F_t^{\vartheta,-}(\underline \eta) \hcht'(\underline \eta) \quad \text{and} \quad \partial_{\overline \eta} \Phi_t^\vartheta(\eta) = F_t^{\vartheta,+}(\overline \eta) \hdst'(\overline \eta).
\end{equation*}
\eqref{eq:activation-lower-bound} yields
\begin{align}
\Phi_t^\vartheta(\vartheta_t^*)
-
\Phi_t^\vartheta(\eta)
&\le\frac1\alpha\left\langle\nabla_\eta\Phi_t^\vartheta(\eta), \vartheta_t^*-\eta\right\rangle- \frac{\kappa}{2}\norm{\vartheta_t^* - \eta}_2^2\nonumber\\
&\le \max_{\ths_t \in \thsspace_t}\left\{ \frac1\alpha\left\langle\nabla_\eta\Phi_t^\vartheta(\eta), \ths_t-\eta\right\rangle- \frac{\kappa}{2}\norm{\ths_t - \eta}_2^2 \right\}.
\label{eq:one-stage-KL}
\end{align}
Inequality \eqref{eq:one-stage-KL} is the gradient dominance property from \cite[Definition 2]{chen2026LandscapePolicyOptimization}. With \cite[Lemma 16]{chen2026LandscapePolicyOptimization}
we get the \PLK inequality of $\Phi^\ths_t$ with \PLK-constant
\begin{equation}
 \mu_Q=\kappa\alpha^2.
\label{eq:mu-Q}
\end{equation}

We finally establish the sequential-decomposition estimate.
Fix \(t<k\) and let
\[
\vartheta^\alpha
=
(\vartheta_{<k},\vartheta_k,\vartheta_{>k}^*) 
\]
and
\[
\vartheta^\beta
=
(\vartheta_{<k},\vartheta_k^*,\vartheta_{>k}^*).
\]
The two policies have the same prefix before time \(k\), and they have the same optimal continuation after time \(k\).  In particular, \(a_k^{\vartheta^\alpha}=a_k^{\vartheta^\beta}=h_k'\), \(\underline a_k^{\vartheta^\alpha}=\underline a_k^{\vartheta^\beta}=\underline h_k'\) and \(\overline a_k^{\vartheta^\alpha}=\overline a_k^{\vartheta^\beta}=\overline h_k'\) .  
With \eqref{eq: inequality sup-norm difference in a} we get, for $t < k$,
\begin{align*}
 \sup_y|a^{\ths^\alpha}_t(y) - a^{\ths^\beta}_t(y)|
 &\le L_Z \left(
 \int_{\min\{\uth_k,\uth_k^*\}}^{\max\{\uth_k,\uth_k^*\}}
   |\hch{k}'(x)|\,dx+
 \int_{\min\{\oth_k,\oth_k^*\}}^{\max\{\oth_k,\oth_k^*\}}
   |\hds{k}'(x)|\,dx \right).
\end{align*}
Because the two strongly concave functions $\hch{k}$ and $\hds{k}$ are increasing before and decreasing after
their respective maximizers $\lt_k^*$ and $\ut_k^*$, the last two integrals equal the corresponding
suboptimality gaps.  Hence
\[
 \sup_y|a^{\ths^\alpha}_t(y) - a^{\ths^\beta}_t(y)| \le L_Z\Delta_k,
 \qquad
 \Delta_k\coloneqq\hch{k}(\uth_k^*)-\hch{k}(\uth_k)
          +\hds{k}(\oth_k^*)-\hds{k}(\oth_k).
\]
 Since the state distribution at time \(t<k\) is also identical
under \(\vartheta^\alpha\) and \(\vartheta^\beta\), formulas
\eqref{eq:grad-lower}--\eqref{eq:grad-upper} yield the two-dimensional bound
\begin{align}
\norm{
\nabla_{\vartheta_t}\ell(\vartheta^\alpha)
-
\nabla_{\vartheta_t}\ell(\vartheta^\beta)
}_2
\le{}&
2L_Z \Delta_k.
\label{eq:sequential-1}
\end{align}

We now compare \(\Delta_k\) with the one-stage loss.  At \(X_k=0\), the difference
between the optimal and the current time-\(k\) action is exactly
\(\hch{k}(\uth_k^*)-\hch{k}(\uth_k)\).  At \(X_k=\PufferStoremax\), the corresponding difference is
exactly \(\hds{k}(\oth_k^*)-\hds{k}(\oth_k)\).  At every other state the difference is
nonnegative because \(\vartheta_k^*\) is optimal for \(Q_k\).  Since
\(\PP_\rho^\vartheta(X_k=0)\ge\alpha\) and
\(\PP_\rho^\vartheta(X_k=\PufferStoremax)\ge\alpha\),
\begin{equation}
 \Phi_k^\vartheta(\vartheta_k^*)-\Phi_k^\vartheta(\vartheta_k)
 \ge\alpha\Delta_k.
 \label{eq:C20}
\end{equation}
Combining \eqref{eq:sequential-1} and \eqref{eq:C20},
\begin{equation}
 \norm{
  \nabla_{\vartheta_t}\ell(\vartheta^\alpha)
  -\nabla_{\vartheta_t}\ell(\vartheta^\beta)
 }_2
 \le\frac{2L_Z}{\alpha}
 \bigl(
   \Phi_k^\vartheta(\vartheta_k^*)-\Phi_k^\vartheta(\vartheta_k)
 \bigr).
 \label{eq:C21}
\end{equation}
Thus the sequential-decomposition constant can be chosen as
\begin{equation}
M_g
=
\frac{2L_Z}{\alpha}.
\label{eq:Mg}
\end{equation}

For completeness, the finite-horizon performance-difference identity used in the general
KL theorem can be seen directly by a telescoping argument.  Define the hybrid policies
\[
 \vartheta^{[k]}
 \coloneqq(\vartheta_0,\dots,\vartheta_{k-1},
    \vartheta_k^*,\dots,\vartheta_{T-1}^*),
 \qquad k=0,\dots,T.
\]
Then \(\vartheta^{[0]}=\vartheta^*\),
\(\vartheta^{[T]}=\vartheta\), and the \(k\)-th difference in the telescoping sum changes
only the decision at time \(k\), while the prefix distribution is generated by
\(\vartheta_{<k}\) and the continuation is optimal.  Hence
\begin{equation}
 \ell^*-\ell(\vartheta)
 =\sum_{t=0}^{T-1}
   \bigl(
    \Phi_t^\vartheta(\vartheta_t^*)-\Phi_t^\vartheta(\vartheta_t)
   \bigr).
\label{eq:C23}
\end{equation}

We have now verified the three hypotheses of the finite-horizon \PLK composition theorem
of \cite[Theorem 1]{chen2026LandscapePolicyOptimization}: the block gradients are bounded by \(G\), the
one-stage expected optimal \(Q\)-value functions satisfy the constrained \PLK inequality
with constant \(\mu_Q\), and the sequential-decomposition inequality holds with constant
\(M_g\).  Applying the maximization version of that theorem (equivalently, applying its
minimization statement after changing signs) gives
\[
 \ell^*-\ell(\vartheta)
 \le \frac{1}{2\mu_\ell}
    \min_{n \in N_\thsspace(\ths)}\norm{\nabla\ell(\vartheta)-n}^2,
\]
where
\[
 \mu_\ell=\frac{\mu_Q^3}{eM_g^2G^2T^2}.
\]

Substituting \eqref{eq:mu-Q}, \eqref{eq:G}, and \eqref{eq:Mg} yields
\[
\mu_\ell
=
\frac{
\kappa^3\alpha^8
}{
16eL_Z^2\psum^2T^2
}.
\]
This proves  \cref{lem:KL}.
\end{proof}

\begin{proof}[Proof of  \cref{thm:PG}]
The first assertion
\begin{equation}
 \E\bigl[\ell^*-\ell(\vartheta^{(K)})\bigr]
 \le
 \left(1-\frac{\mu_\ell}{16L_{\nabla \lo}+\mu_\ell}\right)^K
 \bigl(\ell^*-\ell(\vartheta^{(0)})\bigr)
 +\frac{17\sigma^2}{\mu_\ell N},
 \label{eq:47}
\end{equation}
follows from \cite[Lemma 1]{chen2026LandscapePolicyOptimization} for $f \coloneqq -\lo$. The necessary properties of $\lo$ and $\gradientsmp \lo_N$ were shown in \cref{lem:unbiased,lem:smooth,lem:KL}.

  If
\begin{equation}
 N\ge\frac{34\sigma^2}{\mu_\ell\delta q}
 =\frac{68\psum^2T}{\mu_\ell\delta q},
 \label{eq:48}
\end{equation}
then the second term on the right-hand side of \eqref{eq:47} is at most
\(\delta q/2\).  If \(\Delta_0>0\) and
\begin{equation}
 K\ge\frac{16L+\mu_\ell}{\mu_\ell}
       \log\!\left(\frac{2\Delta_0}{\delta q}\right),
 \label{eq:49}
\end{equation}
then, using \(1-x\le e^{-x}\), the first term is at most \(\delta q/2\).
If \(\Delta_0=0\), the claim is already trivial.  Hence in all cases
\[
 \E\bigl[\ell^*-\ell(\vartheta^{(K)})\bigr]\le\delta q.
\]
Since the random variable \(\ell^*-\ell(\vartheta^{(K)})\) is nonnegative, Markov's
inequality gives
\begin{equation}
 \PP\bigl(\ell^*-\ell(\vartheta^{(K)})>\delta\bigr)
 \le q.
 \label{eq:410}
\end{equation}
This proves  \cref{thm:PG}.
\end{proof}

Proof of \cref{cor:threshold_bounds}.

\begin{proof}
  Recall from the proof of \cref{lem:KL} that the functions
$\hcht$ and $\hdst$ are $\kappa$-strongly concave on
$[0,P_{\max}]$, with respective maximizers
$\underline\vartheta_t^*$ and $\overline\vartheta_t^*$.
Hence
\[
 \hcht(\underline\vartheta_t^*)
   -\hcht(\underline\vartheta_t)
 \ge
 \frac{\kappa}{2}
 |\underline\vartheta_t-\underline\vartheta_t^*|^2
\]
and
\[
 \hdst(\overline\vartheta_t^*)
   -\hdst(\overline\vartheta_t)
 \ge
 \frac{\kappa}{2}
 |\overline\vartheta_t-\overline\vartheta_t^*|^2.
\]
Together with \cref{eq:C20}, this yields
\[
 \Phi_t^\vartheta(\vartheta_t^*)
 -\Phi_t^\vartheta(\vartheta_t)
 \ge
 \frac{\alpha\kappa}{2}
 \|\vartheta_t-\vartheta_t^*\|_2^2.
\]
Using the performance-difference identity \eqref{eq:C23}, we obtain
\[
 \ell^*-\ell(\vartheta)
 \ge
 \frac{\alpha\kappa}{2}
 \|\vartheta-\vartheta^*\|_2^2.
\]  
In particular, $\vartheta^*$ is the unique global maximizer of
$\ell$ under  \cref{ass:PG}. Therefore,
\[
 \E\|\vartheta^{(K)}-\vartheta^*\|_2^2
 \le
 \frac{2}{\alpha\kappa}
 \E\bigl[\ell^*-\ell(\vartheta^{(K)})\bigr].
\]
Applying  \cref{thm:PG} and using
$\kappa=\pem \mu_Z$
gives \cref{eq:threshold_difference}.
\end{proof}

\newpage
\printbibliography

\end{document}